\documentclass[11pt]{article}
\usepackage[margin=1in]{geometry}
\usepackage{algorithm}
\usepackage{algpseudocode}
\usepackage{url} 

\providecommand{\headers}[2]{}

\let\oldtitle\title
\renewcommand{\title}[2][]{\oldtitle{#2}}

\makeatletter
\newcommand{\LoadPackageUnlessAcm}[2][]{%
    \@ifclassloaded{acmart}{%
        \PackageInfo{preamble}{Skipping package `#2' because acmart class is used}%
    }{%
        \if\relax\detokenize{#1}\relax
            \usepackage{#2}%
        \else
            \usepackage[#1]{#2}%
        \fi
    }%
}
\makeatother

\LoadPackageUnlessAcm{caption}
\LoadPackageUnlessAcm{subcaption}
\LoadPackageUnlessAcm{authblk}
\LoadPackageUnlessAcm{hyperref}

\usepackage{amstext}
\usepackage{stmaryrd}
\usepackage{algpseudocode}

\usepackage{units}
\usepackage{cancel}
\usepackage{graphicx}

\usepackage{xcolor}
\usepackage{array}
\usepackage{verbatim}
\usepackage{booktabs}
\usepackage{makecell}
\usepackage{calc}
\usepackage{textcomp}
\usepackage{multirow}
\usepackage{tikz,tikzscale}

\usepackage{pgfplots}
\usepackage{pgfplotstable}
\usepackage{pgfkeys}
\pgfplotsset{compat=1.18}

\definecolor{deepgreen}{RGB}{0,100,0}

\newif\ifrunscripts
\runscriptstrue
\pgfplotsset{
    p2/.style={teal, mark=*},
    p3/.style={orange, mark=triangle*},
    p7/.style={blue, mark=square*},
}

\newcommand{\plotwidth}{0.49\columnwidth}
\newcommand{\plotheight}{0.4\columnwidth}

\pgfplotscreateplotcyclelist{paper markers}{
    {teal,    solid, mark=*},
    {orange,   solid, mark=triangle*},
    {blue,     solid, mark=square*},
    {red,      solid, mark=pentagon*},
    {violet,     solid, mark=otimes*},
    {brown,    solid, mark=halfcircle*},
    {black,   solid, mark=diamond*},
}

\colorlet{BarColorOne}{orange!60!white}
\colorlet{BarColorTwo}{teal!80!white}
\colorlet{BarColorThree}{red!80!black}
\colorlet{BarColorFour}{blue!60!white}
\colorlet{BarColorFive}{brown!70!black}

\pgfplotsset{legend swatch/.style={area legend, draw=none}}

\pgfplotsset{
    paperplot/.style={
            width=\plotwidth,
            height=\plotheight,
            cycle list name=paper markers,
            every axis/.append style={font=\small},
            grid=major,
            thick
        }
}

\pgfplotsset{
    longplot/.style={
            width=0.8\columnwidth,
            height=\plotheight,
            cycle list name=paper markers,
            every axis/.append style={font=\small},
            grid=major,
            thick
        }
}

\usepackage[all]{xy}
\usepackage{mathtools}
\usepackage{placeins}
\usetikzlibrary{calc}
\usetikzlibrary{arrows.meta}
\usetikzlibrary{shadings}
\usepgfplotslibrary{fillbetween}

\usepackage[mode=multiuser,status=draft]{fixme}
\fxsetup{innerlayout=inline}
\fxsetup{targetlayout=colorcb}
\fxusetheme{color}
\FXRegisterAuthor{mw}{envmw}{MW}

\definecolor{apply}{rgb}{0.3,.7,0.}
\definecolor{applyface}{rgb}{.9,.95,0.}
\definecolor{invert}{rgb} {0.5,0.,1.}
\colorlet{invertface}{invert!60!red}

\usepackage[colorinlistoftodos,prependcaption,textsize=small]{todonotes}

\usepackage{colortbl}

\usepackage{booktabs}
\usepackage[font=footnotesize,position=b]{subcaption}
\usepackage{adjustbox}

\makeatletter
\def\pgfplots@getautoplotspec into#1{%
        \begingroup
        \let#1=\pgfutil@empty
        \pgfkeysgetvalue{/pgfplots/cycle multi list/@dim}\pgfplots@cycle@dim
        \let\pgfplots@listindex=\pgfplots@numplots
        \pgfkeysgetvalue{/pgfplots/cycle list set}\pgfplots@listindex@set
        \ifx\pgfplots@listindex@set\pgfutil@empty
        \else
            \c@pgf@counta=\pgfplots@listindex
            \c@pgf@countb=\pgfplots@listindex@set
            \advance\c@pgf@countb by -\c@pgf@counta
            \globaldefs=1\relax
            \edef\setshift{%
                \noexpand\pgfkeys{
                    /pgfplots/cycle list shift=\the\c@pgf@countb,
                    /pgfplots/cycle list set=
                }
            }%
            \setshift%
            \globaldefs=0\relax
        \fi
        \pgfkeysgetvalue{/pgfplots/cycle list shift}\pgfplots@listindex@shift
        \ifx\pgfplots@listindex@shift\pgfutil@empty
        \else
            \c@pgf@counta=\pgfplots@listindex\relax
            \advance\c@pgf@counta by\pgfplots@listindex@shift\relax
            \ifnum\c@pgf@counta<0
                \c@pgf@counta=-\c@pgf@counta
            \fi
            \edef\pgfplots@listindex{\the\c@pgf@counta}%
        \fi
        \ifnum\pgfplots@cycle@dim>0
            \c@pgf@counta=\pgfplots@cycle@dim\relax
            \c@pgf@countb=\pgfplots@listindex\relax
            \advance\c@pgf@counta by-1
            \pgfplotsloop{%
                \ifnum\c@pgf@counta<0
                    \pgfplotsloopcontinuefalse
                \else
                    \pgfplotsloopcontinuetrue
                \fi
            }{%
                \pgfkeysgetvalue{/pgfplots/cycle multi list/@N\the\c@pgf@counta}\pgfplots@cycle@N
                \pgfplotsmathmodint{\c@pgf@countb}{\pgfplots@cycle@N}%
                \divide\c@pgf@countb by \pgfplots@cycle@N\relax
                \expandafter\pgfplots@getautoplotspec@
                \csname pgfp@cyclist@/pgfplots/cycle multi list/@list\the\c@pgf@counta @\endcsname
                {\pgfplots@cycle@N}%
                {\pgfmathresult}%
                \t@pgfplots@toka=\expandafter{#1,}%
                \t@pgfplots@tokb=\expandafter{\pgfplotsretval}%
                \edef#1{\the\t@pgfplots@toka\the\t@pgfplots@tokb}%
                \advance\c@pgf@counta by-1
            }%
        \else
            \pgfplotslistsize\autoplotspeclist\to\c@pgf@countd

            \pgfplots@getautoplotspec@{\autoplotspeclist}{\c@pgf@countd}{\pgfplots@listindex}%
            \let#1=\pgfplotsretval
        \fi
        \pgfmath@smuggleone#1%
        \endgroup
    }

\pgfplotsset{
    cycle list set/.initial=
}
\makeatother

\usepackage{lipsum}
\usepackage{amsfonts}
\usepackage{graphicx}
\usepackage{epstopdf}
\usepackage{amsmath,amssymb,amsthm}
\newtheorem{lemma}{Lemma}
\newtheorem{theorem}{Theorem}
\newtheorem{corollary}{Corollary}
\theoremstyle{definition}
\newtheorem{assumption}{Assumption}
\theoremstyle{remark}

\theoremstyle{plain}
\ifpdf
    \DeclareGraphicsExtensions{.eps,.pdf,.png,.jpg}
\else
    \DeclareGraphicsExtensions{.eps}
\fi

\newcommand{\secref}[1]{Section~\ref{#1}}

\headers{Matrix-Free Finite Elements in Cell-Wise Storage}{Wichrowski}

\title{Coalesced Matrix-Free Finite Elements in Cell-Wise Storage}
\author{Michał Wichrowski}

\usepackage{amsopn}

\providecommand{\keywordsname}{Keywords}
\newcommand{\Keywords}[1]{\par\noindent\textbf{\keywordsname:} #1}
\newcommand{\AMS}[1]{\par\noindent\textbf{AMS subject classifications:} #1}

\title{Coalesced Matrix-Free Geometric Multigrid on Persistent Cell-Wise Storage}

\begin{document}
\author{Michał Wichrowski}
\footnotetext{Interdisziplinäres Zentrum für Wissenschaftliches Rechnen (IWR), Ruprecht-Karls-Universität Heidelberg, Germany, \texttt{mwichro@mimuw.edu.pl}}
\date{}

\maketitle

\begin{abstract}
    %!TEX root = ../main_mg.tex
We present a geometric multigrid preconditioner for high-order continuous finite elements that operates entirely on
redundant, cell-wise stored vectors: the assembled global vector is never formed on any level of the hierarchy. In this
storage paradigm the machinery that classically complicates adaptive multigrid dissolves. Hanging-node constraints are
never assembled: we prove that the plain tensor-product transfer operators, applied to the \emph{unassembled} residual,
algebraically reproduce the classical constrained restriction, including the action of the transposed constraint
matrix, and the edge operators of local smoothing reduce to a pointwise masking of the residual, with no splitting of
the level operator into interior and edge blocks. As a consequence, the single inter-cell primitive of the whole
V-cycle can use a topologically oblivious structured kernel even on adaptively refined meshes. We prove that the
resulting cell-wise V-cycle is equivalent, iterate by iterate, to the classical local multigrid method, and therefore
inherits its convergence theory. Numerical experiments for the Laplace operator confirm grid-independent convergence
that is essentially unaffected by local refinement; on a single GPU, using nothing more than a masked point-Jacobi
smoother, the solver sustains up to $1.1$\, GDoF/s per V-cycle in double precision and reaches end-to-end solve
throughput on par with patch-smoother-based solvers.

\end{abstract}

\Keywords{geometric multigrid, adaptive mesh refinement, hanging nodes, matrix-free, finite elements, GPU}

\AMS{65N55, 65N30, 65F08, 65Y10, 65Y20}

%%%%%%%%%%%%%%%%%%%%%%%%
% Introduction — to be written.
%%%%%%%%%%%%%%%%%%%%%%%%
\section{Introduction}
\label{sec:introduction}
%% TODO: write the introduction. Selling points to lead with:
%%  1. hanging-node constraints no longer needed — embedded in the MG structure;
%%  2. novel realization of edge operators (dual-side masking, no A^SE/A^ES blocks);
%%  3. DSS works unchanged on adaptively refined meshes;
%%  4. novel transfers: valence weighting removed entirely (restrict the unassembled dual vector).
%% Headline number: at p=3, masked point-Jacobi MG matches a patch-smoother solver in solve time (fp64 vs fp64).

%!TEX root = ../main_mg.tex

%%%%%%%%%%%%%%%%%%%%%%%%
% Motivation: high-order matrix-free MG on GPUs; constraints are a data-movement problem.
%%%%%%%%%%%%%%%%%%%%%%%%
% Paragraph: setting — matrix-free high-order FEM on GPUs, multigrid as the solver of choice.
Geometric multigrid is the natural solver for high-order finite element discretizations of elliptic problems: with
sum-factorized operator evaluation and tensor-product inter-grid transfers, the cost of a V-cycle scales as
$\mathcal{O}(p^{d+1})$ per cell and level-independent convergence is well established for a broad class of
problems~\cite{Hackbusch85,Bramble93,ronquist1987spectral,sundar2015comparison}. On modern GPUs, where virtually all
element-wise operations are limited by data movement rather than arithmetic~\cite{williams2009roofline}, the
matrix-free realization of such solvers---including the block-structured hierarchical-hybrid-grid
solvers~\cite{bergen2004hierarchical,gmeiner2015towards} and the hybrid geometric/polynomial hierarchies used at
scale~\cite{fehn2020hybrid,kronbichler2023enhancing}---has been refined to the point that the sum-factorized kernels
themselves are no longer the bottleneck: the cost is concentrated in how data is stored and moved between
them~\cite{fischer2020scalability,kronbichler2022cg}.

% Paragraph: adaptivity introduces hanging nodes; constraints permeate the solver; on GPUs this is a performance problem.
Adaptive mesh refinement compounds this problem. On locally refined quadrilateral and hexahedral meshes, conformity is
classically enforced by constraining the \emph{hanging nodes} that appear along boundaries between regions of different
refinement depth to interpolate their conforming neighbors, and these constraints reach into every component of the
solver. In an assembled framework they are built into the assembly loop, eliminate rows and columns of the system
matrix, and are applied explicitly during every operator evaluation, every smoother step, and every inter-grid
transfer~\cite{bangerth2012algorithms,Kronbichler2012}. Modern matrix-free solvers avoid forming the stiffness matrix
but retain the full constraint infrastructure: index sets, constraint matrices, and the splitting of level operators
into interior and edge blocks required by the local smoothing
algorithm~\cite{JanssenKanschat11,kronbichler2019multigrid,munch2023cache,ljungkvist2017matrix}. On a GPU this infrastructure is not merely
an implementation nuisance. Constraint resolution is realized through indirect, data-dependent memory accesses embedded
in the gather-scatter stage---exactly the access pattern that forfeits coalescing and fragments the high-throughput
sweeps that the structured parts of the mesh would otherwise admit. The handling of a lower-dimensional set of
interface degrees of freedom thereby shapes---and slows---the kernels that process the volume.

% Paragraph: prior work — local smoothing lineage; what none of them eliminate.
The theoretical foundations of adaptive geometric multigrid are well developed. Restricting the smoother to the locally
refined part of each level---\emph{local smoothing}---goes back to Brandt's multilevel adaptive
technique~\cite{brandt1977mlat} and the fast adaptive composite-grid method of McCormick and
Thomas~\cite{mccormick1986fac}; level-independent convergence in the conforming finite element setting with
hanging-node constraints was established by Janssen and Kanschat~\cite{JanssenKanschat11} in the multilevel
subspace-correction framework~\cite{bramble1990parallel,Bramble93}; uniform convergence on adaptively refined meshes
was also shown by Wu and Chen~\cite{wu2006uniform}. Their algorithm, implemented in deal.II~\cite{Kanschat08smoother},
was subsequently extended to massively parallel and matrix-free
settings~\cite{sundar2012parallel,clevenger2021flexible,kronbichler2019multigrid,munch2023cache,wichrowski2022matrix,bergbauer2025high} and
serves as the standard reference for high-order adaptive multigrid. What all of these works share---what none of them
eliminate---is the treatment of hanging-node constraints as explicit algebraic objects: the local smoother must know
the refinement-edge partition, the inter-grid residual transfer is mediated by the transposed constraint matrix, and
the implementation queries the constraint data structures every time the operator touches a non-conforming interface.

%%%%%%%%%%%%%%%%%%%%%%%%
% This paper: the cell-wise framework and what the multigrid construction achieves on top of it.
%%%%%%%%%%%%%%%%%%%%%%%%
% Paragraph: starting point — the cell-wise storage framework of the companion paper.
This paper constructs an adaptive geometric multigrid preconditioner in which none of this machinery appears. The
starting point is the cell-wise storage framework of the companion paper~\cite{wichrowski2026DSS}: every field is kept
in a redundant, element-local representation throughout the entire solver, and the assembled vector of unique degrees
of freedom is never formed. Operator evaluation, vector algebra, and the Krylov dot products are then purely
cell-local; the only operation that communicates between cells is direct stiffness summation (DSS), which accumulates
redundantly stored partial integrals across shared interfaces. The companion paper develops a GPU-efficient DSS
algorithm for block-structured meshes that requires neither indirect addressing nor atomic operations, sweeping the
block-structured memory layout with statically known offsets instead of traversing connectivity tables.

% Paragraph: contribution 1 — constraint-free, valence-free hanging-node treatment, proved exact.
The multigrid construction built on this framework has three properties that, taken together, distinguish it from the
methods above. First, \emph{hanging-node constraints are entirely absent}---not approximated, not hidden in a wrapper,
but never represented: at no point does any solver component evaluate, store, or enforce a constraint across a
non-conforming interface. Conformity is carried implicitly by the geometric inter-grid transfers. The tensor-product
prolongation interpolates a coarse field onto child cells and produces a $C^0$-continuous result by construction, so no
constraint matrix is needed to restore continuity after prolongation; the residual restriction is its adjoint acting on
the \emph{unassembled} residual, and we prove that it reproduces exactly the classical constrained transfer---the
transposed constraint matrix applied to the assembled residual. The unassembled representation is essential here:
because the partial interface integrals remain naturally partitioned among the sharing cells, the restriction routes
them to the correct coarse degrees of freedom without any topology-dependent interface weights. The equivalence is a
consequence of the primal-dual structure of the storage, not an extra algorithmic step.

% Paragraph: contribution 2 (centerpiece) — the convergence mask is the same mechanism that licenses the oblivious kernel.
Second---and this is the observation we consider the conceptual core of the paper---\emph{the requirements of the
    convergence theory and the requirements of the fast kernel coincide at the refinement edge}. The structured DSS cascade
of~\cite{wichrowski2026DSS} enumerates its exchange partners by memory layout, not by topology; on the level meshes of
an adaptive hierarchy it therefore exchanges data across hanging interfaces with cells that have no active counterpart
at the current level, producing incorrect sums there. One might expect this to force a topology-aware variant,
reintroducing the per-interface queries the structured design eliminates. It does not, for the following reason: the
local multigrid method must, for the convergence reasons established in~\cite{JanssenKanschat11}, zero the residual on
the refinement edge before the smoother's preconditioner acts and discard the smoother's correction there afterwards.
These two masking operations render the cascade's errors harmless---the incorrect sums are produced from entries that
are zeroed on input and land in entries that are discarded on output, and they never enter the iterate. The mechanism
required for convergence and the mechanism that licenses the topologically oblivious kernel are the same mechanism. As
a result, the highly optimized conforming DSS kernel runs \emph{unchanged} on arbitrarily complex adaptive refinement
patterns.

% Paragraph: contribution 3 — communication confined to the smoother; abstract smoother requirement.
Third, the entire V-cycle---level-wise residual distribution, restriction, prolongation, operator application, and the
dot products of the outer conjugate gradient accelerator---involves no inter-cell
communication except for the DSS application inside the smoother. The theory makes only one structural demand on the
smoother: it must consume the residual through its assembled form and return a continuity-satisfying correction. DSS
followed by diagonal Jacobi scaling is the simplest realization of this requirement and is the variant studied here;
overlapping tensor-product patch smoothers~\cite{pavarino1993additive,WitteArndtKanschat21,brubeck2021scalable,witte2025tensor,cui2025implementation}, whose
patch-local solves produce continuous corrections by construction, satisfy the same requirement and can be substituted
without modifying any other component of the hierarchy.

% Paragraph: theory by equivalence, and honest scope.
We do not derive convergence estimates from scratch. Instead, we prove that the cell-wise V-cycle is algebraically
identical, iterate for iterate, to the local multigrid V-cycle of Janssen and Kanschat~\cite{JanssenKanschat11}: the
equivalence is an exact operator identity, established by induction over the levels from three intertwining
relations---for the prolongation, for the restriction of the unassembled residual, and for the masked smoother. The
level-independent convergence theory of~\cite{JanssenKanschat11} then transfers verbatim. The scope of the claim should
be stated plainly: the contribution of this paper is architectural, not spectral. With the point-Jacobi smoother the
iteration counts grow with the polynomial degree, as they do for any point smoother, and we make no claim of
$p$-robustness; $p$-robust contraction requires the subspace-correction or patch smoothers cited
above~\cite{schoberl2008additive,pazner2020efficient}. The experiments of \secref{sec:results} situate the cheap,
perfectly structured Jacobi sweep studied here against a state-of-the-art patch-smoother
implementation~\cite{cui2025implementation}.

% Paragraph: implementation vehicle and results preview.
The solver is implemented in Triton~\cite{triton2019}, a tile-based GPU language whose programming model aligns with
the block-structured cell-wise layout and whose compiler handles memory coalescing and tensor-core mapping without
vendor-specific code. On a single A100, the resulting V-cycle sustains up to $1.1$\,GDoF/s per cycle at $p = 3$ in
\texttt{fp64}, for an end-to-end solve throughput of about $100$\,MDoF/s---on par with a recent vertex-patch smoother
implementation~\cite{cui2025implementation}, while driving the error to a tighter tolerance---and both the iteration
count and the streaming throughput are essentially unchanged under deep adaptive refinement.

%%%%%%%%%%%%%%%%%%%%%%%%
% Outline
%%%%%%%%%%%%%%%%%%%%%%%%
% Paragraph: roadmap.
The remainder of the paper is organized as follows. \secref{sec:framework} recalls the cell-wise framework and
notation of the companion paper~\cite{wichrowski2026DSS}. \secref{sec:multigrid} constructs the V-cycle, first on
uniformly refined hierarchies and then in the locally refined setting, and proves the equivalence with classical local
multigrid. \secref{sec:smoother_realization} supplies the algorithmic realization of the smoother---the DSS-based
Jacobi sweep---and analyzes the interaction between the structured DSS cascade and the refinement edge.
\secref{sec:implementation} records the implementation measures that reduce a V-cycle to a minimal number of
streaming passes. \secref{sec:results} presents numerical experiments, and \secref{sec:conclusions}
concludes.

%%%%%%%%%%%%%%%%%%%%%%%%
% Model problem and discretization 
%%%%%%%%%%%%%%%%%%%%%%%%
\section{Preliminaries: the Cell-Wise Solver Framework}
\label{sec:framework}

This section collects the notation and the results of the companion paper~\cite{wichrowski2026DSS} that the multigrid
construction builds on. Readers familiar with~\cite{wichrowski2026DSS} may skim it for notation; all proofs are given
there.

% Paragraph: model problem.
As a model problem we consider the Poisson equation on a bounded domain $\Omega \subset \mathbb{R}^d$, $d \in \{2,3\}$,
with homogeneous Dirichlet boundary conditions: find $u \in H_0^1(\Omega)$ such that
\begin{equation} \label{eq:poisson}
    \int_\Omega \nabla u \cdot \nabla v \, dx = \int_\Omega f v \, dx \qquad \text{for all } v \in H_0^1(\Omega).
\end{equation}
The domain is partitioned into a mesh $\mathcal{T}$ of quadrilateral ($d=2$) or hexahedral ($d=3$) cells, and the
discretization uses continuous tensor-product finite elements $\mathbb{Q}_p$ of polynomial degree $p$ with
Gauss--Lobatto nodal bases. Everything below applies verbatim to any symmetric second-order elliptic operator; the Laplacian is used for the numerical experiments of \secref{sec:results}.

%%%%%%%%%%%%%%%%%%%%%%%%
% Cell-wise storage, primal/dual spaces, gather and constraint operators
%%%%%%%%%%%%%%%%%%%%%%%%
\subsection{Cell-Wise Storage and the Primal-Dual Structure}
\label{sec:cellwise_recap}

% Paragraph: the two spaces and the gather operator.
The framework uses a cell-wise storage format: every field is kept in the redundant, \emph{cell-wise} (element-local)
representation throughout the entire solver, and the assembled vector of unique degrees of freedom is never formed. Let
$\mathbb{V}_{\text{cell}}$ denote the broken space of all local degrees of freedom, stored independently on each cell,
and let $\mathbb{V}_{CG}$ denote the minimal global space of the conforming discretization. The \emph{gather} operator
\begin{equation}
    \mathcal{G} \colon \mathbb{V}_{CG} \to \mathbb{V}_{\text{cell}}
\end{equation}
replicates each global coefficient into all cell-local copies of the corresponding degree of freedom; its transpose
$\mathcal{G}^T \colon \mathbb{V}_{\text{cell}}^* \to \mathbb{V}_{CG}^*$ performs the additive assembly across shared
interfaces. A vector $u \in \mathbb{V}_{\text{cell}}$ represents a continuous field when it lies in
$\text{Im}(\mathcal{G})$; equivalently, $u = C u$ for any idempotent constraint operator $C \colon
    \mathbb{V}_{\text{cell}} \to \mathbb{V}_{\text{cell}}$ with $\text{Im}(C) = \mathcal{G}(\mathbb{V}_{CG})$. The
constraint formalism covers hanging-node constraints of adaptively refined meshes in the same way as plain
inter-element continuity, a fact that the multigrid construction of \secref{sec:multigrid} uses.

% Paragraph: primal vs dual objects, the pairing identity.
Two kinds of objects live in the cell-wise storage and must be carefully distinguished. \emph{Primal} vectors
(iterates, search directions, corrections) are continuous fields: their duplicated entries agree by construction.
\emph{Dual} vectors (residuals, fluxes), written with a tilde as $\tilde{r} \in \mathbb{V}_{\text{cell}}^*$, are raw
unassembled cell-local integrals: their duplicated entries hold complementary \emph{partial} sums that are meaningful
without any synchronization. The bridge between the two is the duality pairing identity of~\cite{wichrowski2026DSS}:
for any dual vector $\tilde{f} \in \mathbb{V}_{\text{cell}}^*$ and any continuous primal vector $u = \mathcal{G}
    u^{\text{global}}$,
\begin{equation} \label{eq:pairing_global}
    \langle \tilde{f}, u \rangle = \langle \mathcal{G}^T \tilde{f}, u^{\text{global}} \rangle,
\end{equation}
i.e.\ the local Euclidean pairing of an unassembled dual vector with a continuous primal vector equals the
assembled global pairing. All inner products of a Krylov solver can be arranged to pair one object of each kind, and
are then computable from local data without communication.

% Paragraph: the discrete system in cell-wise form.
The discrete operator is applied cell by cell: the local operator $A \colon \mathbb{V}_{\text{cell}} \to
    \mathbb{V}_{\text{cell}}^*$ maps a continuous primal vector to an unassembled dual vector by matrix-free evaluation of
the element integrals. With the unassembled right-hand side $\tilde{b} \in \mathbb{V}_{\text{cell}}^*$, the discrete
problem reads: find $u \in \mathbb{V}_{\text{cell}}$ with $u = Cu$ such that
\begin{equation} \label{eq:constrained_system}
    C^T A C u = C^T \tilde{b},
\end{equation}
which is equivalent to the classical assembled system $A_{CG} u^{\text{global}} = b_{CG}$ with $A_{CG} =
    \mathcal{G}^T A \mathcal{G}$ and $b_{CG} = \mathcal{G}^T \tilde{b}$, the solutions being related by $u = \mathcal{G}
    u^{\text{global}}$.

%%%%%%%%%%%%%%%%%%%%%%%%
% What the framework demands from the preconditioner/smoother
%%%%%%%%%%%%%%%%%%%%%%%%
% Paragraph: the only structural requirement: produce continuous output from assembled input.
In this framework the operator application $A$ and all vector algebra are cell-local; the only component that converts
a dual residual back into a primal correction is the preconditioner $P \colon \mathbb{V}_{\text{cell}}^* \to
    \mathbb{V}_{\text{cell}}$, and with it the multigrid smoother. These components consume the residual through its
assembly $\mathcal{G}^T \tilde{r}$ and return a continuous vector, admitting the factorization
\begin{equation} \label{eq:precond_consistency}
    P = \mathcal{G} \, P_{CG} \, \mathcal{G}^T
\end{equation}
for some operator $P_{CG} \colon \mathbb{V}_{CG}^* \to \mathbb{V}_{CG}$ on the minimal global space. The multigrid method of \secref{sec:multigrid} is formulated against this requirement,
and the concrete realization using direct stiffness summation is deferred to \secref{sec:smoother_realization}.

\section{Multigrid Preconditioner}
\label{sec:multigrid}

Operator evaluation and DSS assembly lay the foundation for geometric multigrid (MG) preconditioners. To describe the
framework, let us assume a hierarchy of meshes
\begin{gather}
    \mathcal{T}_0 \sqsubset \mathcal{T}_1 \sqsubset \dots \sqsubset \mathcal{T}_L,
\end{gather}
subdividing a domain $\Omega \subset \mathbb{R}^d$. Here, the symbol ``$\sqsubset$'' denotes nestedness, meaning that every cell of a mesh $\mathcal{T}_{\ell}$ is obtained from a cell of the coarser mesh $\mathcal{T}_{\ell-1}$ by refinement. This hierarchy leads to a corresponding sequence of finite element spaces
\begin{gather}
    \mathbb{V}_0 \subset \mathbb{V}_1 \subset \dots \subset \mathbb{V}_L,
\end{gather}
where $\mathbb{V}_\ell$ is the primal space associated with level $\ell$. We identify $\mathbb{V}_\ell$ with its coefficient representation $\mathbb{R}^{\dim \mathbb{V}_\ell}$ and, for simplicity, use the same notation for a finite element function $u_\ell \in \mathbb{V}_\ell$ and its vector of degrees of freedom.

The discrete problem on each level $\ell$ is given by the linear system
\begin{gather}
    \label{eq:matrix_level}
    A_\ell u_\ell = b_\ell,
\end{gather}
where $A_\ell \colon \mathbb{V}_\ell \to \mathbb{V}_\ell^*$ is the discrete operator and $b_\ell \in \mathbb{V}_\ell^*$ is the right-hand side. In the context of multigrid methods, we transition between levels using geometric transfer operators:
\begin{gather}
    \begin{array}{rlcl}
        \mathcal{R}_\ell^{\ell-1} \colon & \mathbb{V}_{\ell}^* & \to & \mathbb{V}_{\ell-1}^*, \\
        \mathcal{P}_{\ell-1}^\ell \colon & \mathbb{V}_{\ell-1} & \to & \mathbb{V}_{\ell},
    \end{array}
\end{gather}
representing restriction and prolongation, respectively. In our framework, $\mathcal{P}_{\ell-1}^\ell$ is typically chosen as the natural embedding between the nested spaces, while $\mathcal{R}_\ell^{\ell-1}$ is defined as its adjoint. Additionally, the algorithm relies on a smoothing operator $\mathcal{S}_\ell$, characterized by its error propagation properties, to reduce high-frequency error components on each level.

\begin{algorithm}[H]
    \caption{Geometric Multigrid $V$-Cycle: $MG(A_\ell, b_\ell, u_\ell, \ell)$}
    \label{alg:mg_vcycle}
    \begin{algorithmic}[1]
        \Require Level $\ell$, current solution $u_\ell$, right-hand side $b_\ell$.
        \If{$\ell = 0$}
        \State $u_0 \gets A_0^{-1} b_0$ \Comment{Coarse Grid Solve}
        \State \Return $u_0$
        \EndIf
        \State $u_\ell \gets \mathcal{S}_\ell^{\nu_1}(A_\ell, b_\ell, u_\ell)$ \Comment{Pre-smoothing}
        \State $r_\ell \gets b_\ell - A_\ell u_\ell$ \Comment{Compute residual}
        \State $b_{\ell-1} \gets \mathcal{R}_\ell^{\ell-1} r_\ell$ \Comment{Restrict residual to coarse grid}
        \State $e_{\ell-1} \gets 0$ \Comment{Initialize coarse error}
        \State $e_{\ell-1} \gets MG(A_{\ell-1}, b_{\ell-1}, e_{\ell-1}, \ell - 1)$ \Comment{Recursive coarse grid correction}
        \State $u_\ell \gets u_\ell + \mathcal{P}_{\ell-1}^\ell e_{\ell-1}$ \Comment{Prolongate and correct}
        \State $u_\ell \gets \mathcal{S}_\ell^{\nu_2}(A_\ell, b_\ell, u_\ell)$ \Comment{Post-smoothing}
        \State \Return $u_\ell$
    \end{algorithmic}
\end{algorithm}

\subsection{Inter-grid Transfer Operations}
% Explain the application of transfer operators via tensor products.
To perform inter-grid transfers between nested hierarchical spaces, we apply the prolongation and restriction
operations utilizing the tensor-product structure of the underlying finite elements. In one dimension, a parent cell is
split into two children, requiring two distinct 1D prolongation matrices, denoted $P_{0}$ and $P_{1}$, corresponding to
the left (child 0) and right (child 1) sub-intervals. For higher-dimensional problems ($d>1$), the local prolongation
for a specific child cell is computed directly via the Kronecker product of the 1D prolongation matrices acting on the
degrees of freedom. The selection of $\mathcal{P}_0$ or $\mathcal{P}_1$ along each spatial dimension is determined by
the corresponding bit of the child's topological index. The restriction operation is similarly defined by applying the
transpose of these 1D matrices.

% Describe the additive nature of the restriction operator on the cell-wise storage.
During restriction, degrees of freedom located on the interface between sibling cells are projected and accumulated
into the shared parent's degrees of freedom. Because multiple fine cells restrict their contributions into the same
coarse parent, the restriction operator is an additive operation over the children. The localized nature of these
operations is detailed in Algorithm~\ref{alg:prolongation} and Algorithm~\ref{alg:restriction}.

% Algorithm for prolongation.
\begin{algorithm}[H]
    \caption{Prolongation in Cell-Wise representation  }
    \label{alg:prolongation}
    \begin{algorithmic}[1]
        \Require Hierarchy of cells, where parent $K$ has children $\{K_c\}$. Coarse values $u_{\ell-1}$. Fine vector $u_\ell$ initialized to zero.
        \Ensure Fine values $u_{\ell}$.
        \For{\textbf{each} coarse parent cell $K$}
        \For{\textbf{each} child $K_c$ of $K$}
        \State $u_{\ell}|_{K_c} \gets u_{\ell}|_{K_c} + \left( \bigotimes_{i=1}^d \mathcal{P}_{b_i} \right) u_{\ell-1}|_{K}$ \Comment{$b_i$ is the $i$-th bit of child index $c$}
        \EndFor
        \EndFor
    \end{algorithmic}
\end{algorithm}

% Algorithm for restriction.
\begin{algorithm}[H]
    \caption{Restriction in Cell-Wise representation }
    \label{alg:restriction}
    \begin{algorithmic}[1]
        \Require Hierarchy of cells, where parent $K$ has children $\{K_c\}$. Fine residual $r_{\ell}$.
        \Ensure Coarse residual $r_{\ell-1}$.
        \For{\textbf{each} coarse parent cell $K$}
        \State Initialize $r_{\ell-1}|_{K} \gets 0$
        \For{\textbf{each} child $K_c$ of $K$}
        \State $r_{\ell-1}|_{K} \gets r_{\ell-1}|_{K} + \left( \bigotimes_{i=1}^d \mathcal{P}_{b_i}^T \right) r_{\ell}|_{K_c}$ \Comment{$b_i$ is the $i$-th bit of child index $c$}
        \EndFor
        \EndFor
    \end{algorithmic}
\end{algorithm}

% Establish the mathematical equivalence of local inter-grid transfers.
To establish that these purely local, cell-wise transfers correctly propagate information between the continuous finite
element spaces, we formulate the relationships between the local operators and their classical, globally-assembled
counterparts.

% Establish the commutative property of the gather and prolongation operators.
\begin{lemma}[Commutativity of Prolongation]
    \label{lem:prolongation_commutativity}
    Let $\mathcal{P}_{CG, \ell-1}^\ell \colon \mathbb{V}_{CG, \ell-1} \to \mathbb{V}_{CG, \ell}$ be the classical global prolongation operator, and $\mathcal{P}_{\text{cell}, \ell-1}^\ell \colon \mathbb{V}_{\text{cell}, \ell-1} \to \mathbb{V}_{\text{cell}, \ell}$ be the local cell-wise prolongation operator implicitly defined in Algorithm~\ref{alg:prolongation}. The gather operations commute with the hierarchical prolongation such that:
    \begin{equation}
        \mathcal{P}_{\text{cell}, \ell-1}^\ell \mathcal{G}_{\ell-1} = \mathcal{G}_{\ell} \mathcal{P}_{CG, \ell-1}^\ell.
    \end{equation}
\end{lemma}

\begin{proof}
    Both composite operators map a globally continuous coarse vector $u_{CG, \ell-1} \in \mathbb{V}_{CG, \ell-1}$ to a redundantly stored fine vector in $\mathbb{V}_{\text{cell}, \ell}$. Because standard finite element shape functions have local support, the polynomial interpolation onto a child cell depends on the degrees of freedom of its immediate parent. Consequently, interpolating the global continuous field and subsequently extracting the local degrees of freedom on the fine grid ($\mathcal{G}_{\ell} \mathcal{P}_{CG, \ell-1}^\ell$) yields the same discrete values as extracting the coarse local degrees of freedom first and evaluating the geometric embedding cell-by-cell ($\mathcal{P}_{\text{cell}, \ell-1}^\ell \mathcal{G}_{\ell-1}$).
\end{proof}

% Establish the algebraic equivalence of local additive restriction of raw dual vectors.
\begin{lemma}[Equivalence of Unassembled Restriction]
    \label{lem:restriction_equivalence}
    Let $\mathcal{R}_{CG, \ell}^{\ell-1} = (\mathcal{P}_{CG, \ell-1}^\ell)^T$ be the classical global restriction operator acting on the assembled dual space, and $\mathcal{R}_{\text{cell}, \ell}^{\ell-1} = (\mathcal{P}_{\text{cell}, \ell-1}^\ell)^T$ be the additive cell-wise restriction acting on unassembled dual vectors (Algorithm~\ref{alg:restriction}). The global assembly of a locally restricted dual vector equals the classical restriction of a globally assembled dual vector:
    \begin{equation}
        \mathcal{G}_{\ell-1}^T \mathcal{R}_{\text{cell}, \ell}^{\ell-1} = \mathcal{R}_{CG, \ell}^{\ell-1} \mathcal{G}_{\ell}^T.
    \end{equation}
\end{lemma}

\begin{proof}
    The relation follows directly from algebraic transposition of Lemma~\ref{lem:prolongation_commutativity}. Taking the transpose of $\mathcal{P}_{\text{cell}, \ell-1}^\ell \mathcal{G}_{\ell-1} = \mathcal{G}_{\ell} \mathcal{P}_{CG, \ell-1}^\ell$ yields:
    \begin{equation*}
        \mathcal{G}_{\ell-1}^T (\mathcal{P}_{\text{cell}, \ell-1}^\ell)^T = (\mathcal{P}_{CG, \ell-1}^\ell)^T \mathcal{G}_{\ell}^T.
    \end{equation*}
    Substituting the definition of the restriction operators as the formal adjoints of the respective prolongation operators gives $\mathcal{G}_{\ell-1}^T \mathcal{R}_{\text{cell}, \ell}^{\ell-1} = \mathcal{R}_{CG, \ell}^{\ell-1} \mathcal{G}_{\ell}^T$.
\end{proof}

Lemma~\ref{lem:restriction_equivalence} highlights an algorithmic advantage of maintaining residuals in their raw,
unassembled state. The partial interface integrals are partitioned among the sharing cells, and the local restriction
routes these fractions to the coarse grid, recovering the exact global integral without any knowledge of the mesh
topology. Restricting a synchronized residual instead (every sharing cell holding the fully accumulated sum) would
over-count each interface integral by the number of sharing cells and would have to be compensated by
topology-dependent weights.

% Conclude the complete equivalence of inter-grid transfers in the cell-wise storage.
\begin{theorem}[Equivalence of Communication-Free Inter-grid Transfers]
    Let $u_{CG, \ell-1} \in \mathbb{V}_{CG, \ell-1}$ be a continuous coarse primal vector with its redundant representation $u_{\ell-1} = \mathcal{G}_{\ell-1} u_{CG, \ell-1}$, and let $\tilde{r}_{\ell} \in \mathbb{V}_{\text{cell}, \ell}^*$ be an unassembled local fine dual residual. The cell-wise transfer operations compute the unassembled representations of their classically transferred global counterparts:
    \begin{align}
        \mathcal{P}_{\text{cell}, \ell-1}^\ell u_{\ell-1}                                & = \mathcal{G}_\ell (\mathcal{P}_{CG, \ell-1}^\ell u_{CG, \ell-1}), \\
        \mathcal{G}_{\ell-1}^T (\mathcal{R}_{\text{cell}, \ell}^{\ell-1} \tilde{r}_\ell) & = \mathcal{R}_{CG, \ell}^{\ell-1} \bar{r}_\ell,
    \end{align}
    where $\bar{r}_\ell = \mathcal{G}_\ell^T \tilde{r}_\ell$ is the globally assembled fine residual.
\end{theorem}

\begin{proof}
    For the primal prolongation, substituting $u_{\ell-1} = \mathcal{G}_{\ell-1} u_{CG, \ell-1}$ into the left-hand side yields $\mathcal{P}_{\text{cell}, \ell-1}^\ell \mathcal{G}_{\ell-1} u_{CG, \ell-1}$. Applying Lemma~\ref{lem:prolongation_commutativity}, this simplifies to $\mathcal{G}_\ell \mathcal{P}_{CG, \ell-1}^\ell u_{CG, \ell-1}$, proving that the cell-wise prolonged vector is the exact gather of the classically prolonged continuous vector.

    For the dual restriction, applying Lemma~\ref{lem:restriction_equivalence} to the unassembled residual $\tilde{r}_\ell$
    directly provides:
    \begin{equation*}
        \mathcal{G}_{\ell-1}^T (\mathcal{R}_{\text{cell}, \ell}^{\ell-1} \tilde{r}_\ell) = \mathcal{R}_{CG, \ell}^{\ell-1} (\mathcal{G}_\ell^T \tilde{r}_\ell).
    \end{equation*}
    Recognizing that $\mathcal{G}_\ell^T \tilde{r}_\ell = \bar{r}_\ell$ proves that the formal assembly of the locally restricted raw residual identically matches the classical restriction of the assembled residual.
\end{proof}

Thus, assembling the fine grid residual prior to restriction is not required; the localized additive restriction embeds
the topological accumulation into the unassembled coarse space.

\subsection{Adaptive  Multigrid and Edge Constraints}

% Paragraph: the adaptive setting: refinement tree, active mesh, hanging nodes.
When local adaptive mesh refinement is introduced, each cell of the macro-grid may be isotropically refined into $2^d$
child cells, and the refinement may be applied recursively to any subset of the children. The recursion produces a
hierarchical refinement tree of the kind maintained by octree-based mesh frameworks~\cite{burstedde2011p4est,
    bangerth2012algorithms}, and the mesh on which the discrete problem is posed (the \emph{active mesh}
$\mathcal{T}_{\text{act}}$) is the set of leaves of this tree. Because refinement is localized, adjacent leaves may
belong to different levels of the tree. Along the boundaries between regions of different refinement depth
(\emph{refinement edges}), the fine cells contribute degrees of freedom that have no counterpart on the coarse side
(hanging nodes). The active mesh is therefore not conforming, and a conforming discretization on it is classically
obtained by constraining every hanging node to interpolate the neighboring coarse degrees of freedom. In assembled
finite element frameworks, these constraints are built into the assembly loop, eliminate rows and columns of the system
matrix, and require dedicated data structures that are queried during operator
application~\cite{bangerth2012algorithms,Kronbichler2012}. In the cell-wise formulation, these constraints do not
appear as explicit objects.

% Paragraph: global multigrid does not fit the adaptive setting.
\paragraph{Local versus global multigrid.}
A classical geometric multigrid method presupposes a sequence of fully refined, conforming meshes $\mathcal{T}_0
    \subset \mathcal{T}_1 \subset \cdots \subset \mathcal{T}_L$ together with the nested finite element spaces
$\mathbb{V}_{CG, 0} \subset \mathbb{V}_{CG, 1} \subset \cdots \subset \mathbb{V}_{CG, L}$, and applies a smoother on
the entire space $\mathbb{V}_{CG, \ell}$ at every level $\ell$~\cite{Hackbusch85,Bramble93}. Such a hierarchy is
unavailable when only a fraction of each level has been refined. One could, in principle, embed the active mesh into a
notional uniformly refined fine mesh, but this is ruled out on two counts: it would introduce degrees of freedom that
are not represented in the discretization at all, and it would make the per-level cost grow with the size of the
uniform fine mesh rather than with the size of the locally refined region. The work of an adaptive V-cycle must be
proportional to the number of \emph{active} cells, summed over the levels.

% Paragraph: local smoothing: the established remedy and its lineage.
The established remedy is \emph{local} multigrid, also known as local smoothing: smoothing on level $\ell$ is
restricted to the subspace spanned by the degrees of freedom of the locally refined region of that level, while the
remainder of the domain is handled by the coarser levels. This idea underlies Brandt's multi-level adaptive
technique~\cite{brandt1977mlat} and the fast adaptive composite-grid (FAC) method of McCormick and
Thomas~\cite{mccormick1986fac}, and its convergence theory was established within the framework of multilevel subspace
correction~\cite{Bramble93}. For the conforming finite element setting with hanging-node constraints, Janssen and
Kanschat~\cite{JanssenKanschat11} formulated the local V-cycle and proved level-independent convergence; their
formulation, implemented in deal.II~\cite{Kanschat08smoother,JanssenKanschat11}, has since been carried to massively
parallel and matrix-free settings~\cite{kronbichler2019multigrid,munch2023cache}. We adopt this method as our classical
reference point: the algorithm constructed below will be shown to be its exact cell-wise realization, and we import its
convergence theory through that equivalence rather than re-deriving it.

% Paragraph: the structure of local smoothing and the two well-posedness questions.
In the subspace-correction reading~\cite{Bramble93}, the discrete problem on $\mathcal{T}_{\text{act}}$ is split into
an additive sum of contributions associated with the cells of each level, and the smoother on level $\ell$ acts only on
the subspace supported by the refined region of that level. Two questions must be addressed for the resulting V-cycle
to be well posed. First, the level-$\ell$ smoother sees data interpolated from the coarser level along $E_\ell$ that it
must not overwrite. In the classical formulation this is handled by treating the edge values as Dirichlet data, which
requires splitting the level operator into interior and edge blocks. Second, the residual must be transferred across
$E_\ell$ from level $\ell$ to level $\ell-1$. In the classical formulation this transfer is mediated by the transposed
hanging-node constraint matrix. The cell-wise formulation eliminates this machinery: the operator splitting reduces to
a pointwise mask on the smoother's residual (\secref{sec:edge_masking}), and the constraint application is performed
implicitly by the plain tensor-product restriction acting on the unassembled residual
(\secref{sec:residual_equivalence}). The remainder of this subsection makes the level-wise decomposition precise in the
cell-wise setting; the two subsections that follow establish the two equivalences. Figure~\ref{fig:local_mg_levels}
illustrates the decomposition.

\begin{figure}[htbp]
    \centering
    % Level-wise decomposition of a locally refined active mesh (2:1 balanced).
% LEFT : the composed active mesh T_act (flat, top-down), cells of three different
%        sizes, refinement interfaces in red.
% RIGHT: its decomposition into the level meshes T_0, ..., T_L drawn one above the
%        other in an oblique exploded view. On each level the cells refined further
%        (inactive there) are shaded gray; active cells are white. The refinement
%        edge between level (ell-1) and ell is drawn on the finer level ell (E_ell,
%        red); the interior region where the level smoother acts is marked (S_ell).
% The level-1 region is L-shaped: a deeply refined block at the bottom-left (down to
% level L) plus a one-level-deep cell at the top-left. The mesh is 2:1 balanced: the
% level-2 region is separated from the level-0 cells by a collar of level-1 cells.
\begin{tikzpicture}[
                scale=1.0,
                % oblique projection: plane coord (u,v) at height z maps to
                % screen (u + \sk*v, \mgd*v + z)
                act/.style={draw=black, thin, fill=white}, inact/.style={draw=black, thin, fill=gray!45}, foot/.style={draw=black!45,
                                thin, dashed}, redge/.style={green!50!black, line width=1.8pt, shorten >=2pt, shorten <=2pt}, lev/.style={font=\large},
        ]% active cell% refined (inactive) cell% domain footprint% refinement edge E_ell

        \def\sk{0.58}   % horizontal shear per unit depth
        \def\mgd{0.42}  % vertical compression per unit depth
        \def\zz{2.05}   % explosion gap between levels
        \def\W{6}       % domain width  (coarse units)
        \def\H{4}       % domain height (coarse units)

        % parallelogram cell of side s with lower-left plane corner (u,v) at height z
        \newcommand{\mgcell}[5]{% {u}{v}{s}{z}{style}
                \filldraw[#5]
                ({#1+\sk*#2},{\mgd*#2+#4})
                -- ({#1+#3+\sk*#2},{\mgd*#2+#4})
                -- ({#1+#3+\sk*(#2+#3)},{\mgd*(#2+#3)+#4})
                -- ({#1+\sk*(#2+#3)},{\mgd*(#2+#3)+#4}) -- cycle;
        }
        % domain footprint [0,W]x[0,H] outline at height z
        \newcommand{\mgfoot}[1]{%
                \draw[foot]
                ({0+\sk*0},{\mgd*0+#1})
                -- ({\W+\sk*0},{\mgd*0+#1})
                -- ({\W+\sk*\H},{\mgd*\H+#1})
                -- ({0+\sk*\H},{\mgd*\H+#1}) -- cycle;
        }
        % a segment on plane between (u1,v1)-(u2,v2) at height z, in given style
        \newcommand{\mgseg}[6]{% {u1}{v1}{u2}{v2}{z}{style}
                \draw[#6] ({#1+\sk*#2},{\mgd*#2+#5}) -- ({#3+\sk*#4},{\mgd*#4+#5});
        }

        % =====================================================================
        % LEFT: composed active mesh T_act (flat, top-down)
        % =====================================================================
        \begin{scope}[shift={(-7.6,1.6)}, scale=0.78]
                % coarse active cells (size 2)
                \foreach \c/\r in {4/0,2/2,4/2}
                \draw[act] (\c,\r) rectangle (\c+2,\r+2);
                % medium active cells (size 1): L-shaped level-1 region (collar of the
                % deep block, plus the once-refined top-left cell)
                \foreach \c/\r in {2/0,3/0,0/1,1/1,2/1,3/1,0/2,1/2,0/3,1/3}
                \draw[act] (\c,\r) rectangle (\c+1,\r+1);
                % fine active cells (size 1/2) in the deeply refined region [0,2]x[0,1]
                \foreach \c in {0,0.5,1,1.5} \foreach \r in {0,0.5}
                \draw[act] (\c,\r) rectangle (\c+0.5,\r+0.5);
                % refinement interfaces
                \draw[redge] (4,0) -- (4,2);                                  % level 0/1
                \draw[redge] (2,2) -- (4,2);   \draw[redge] (2,2) -- (2,4);   % level 0/1
                \draw[redge] (2,0) -- (2,1);   \draw[redge] (0,1) -- (2,1);   % level 1/L
                \node[anchor=north, font=\Large] at (3,-0.25) {$\mathcal{T}_{\mathrm{act}}$};
        \end{scope}

        % arrow from the assembled mesh to the exploded stack
        \draw[-{Latex[length=3mm]}, thick] (-1.9,2.6) -- (-0.6,2.6)
        node[midway, above, font=\scriptsize] {Level split};

        % =====================================================================
        % LEVEL 0 (z=0): coarse mesh, size-2 cells over [0,6]x[0,4]. The bottom-left
        % pair [0,4]x[0,2] and the top-left cell [0,2]x[2,4] are refined -> inactive
        % (gray). No refinement edge on the coarsest level.
        % =====================================================================
        \def\z{0}
        \foreach \c/\r in {4/0,2/2,4/2} \mgcell{\c}{\r}{2}{\z}{act}
        \mgcell{0}{0}{2}{\z}{inact}
        \mgcell{2}{0}{2}{\z}{inact}
        \mgcell{0}{2}{2}{\z}{inact}
        \node[lev, anchor=east] at ({0+\sk*\H-0.35},{\mgd*\H+\z}) {$\mathcal{T}_0$};

        % =====================================================================
        % LEVEL 1 (z=zz): L-shaped region (bottom block [0,4]x[0,2] + top-left cell
        % [0,2]x[2,4]), size-1 cells. The two bottom-left cells [0,2]x[0,1] are
        % refined again -> inactive (gray). The level 0/1 refinement edge (E_1) runs
        % along the staircase boundary with the level-0 cells.
        % =====================================================================
        \def\z{\zz}
        \mgfoot{\z}
        \foreach \c/\r in {2/0,3/0,0/1,1/1,2/1,3/1,0/2,1/2,0/3,1/3} \mgcell{\c}{\r}{1}{\z}{act}
        \mgcell{0}{0}{1}{\z}{inact}
        \mgcell{1}{0}{1}{\z}{inact}
        \mgseg{4}{0}{4}{2}{\z}{redge}          % E_1 along u=4
        \mgseg{2}{2}{4}{2}{\z}{redge}          % E_1 along v=2 (u in 2..4)
        \mgseg{2}{2}{2}{4}{\z}{redge}          % E_1 along u=2 (v in 2..4)
        \node[lev, anchor=east] at ({0+\sk*\H-0.35},{\mgd*\H+\z}) {$\mathcal{T}_1$};
        \node[green!50!black, font=\large, anchor=west] at ({4.05+\sk*1.0},{\mgd*1.0+\z}) {$E_\ell$};

        % =====================================================================
        % LEVEL 2 (z=2*zz): finest region [0,2]x[0,1], size-1/2 cells (4x2), all
        % active. The edge between level 1 and level L is drawn here (E_L).
        % =====================================================================
        \def\z{2*\zz}
        \mgfoot{\z}
        \foreach \i in {0,1,2,3} \foreach \j in {0,1} {
                        \mgcell{\i*0.5}{\j*0.5}{0.5}{\z}{act}
                }
        \mgseg{0}{1}{2}{1}{\z}{redge}          % E_L along v=1
        \mgseg{2}{0}{2}{1}{\z}{redge}          % E_L along u=2
        \node[lev, anchor=east] at ({0+\sk*\H-0.35},{\mgd*\H+\z}) {$\mathcal{T}_L$};
        \node[green!50!black, font=\large, anchor=west] at ({2.15+\sk*0.5},{\mgd*0.5+\z}) {$E_\ell$};

\end{tikzpicture}
    \caption{Level-wise decomposition of a locally refined active mesh. Left: the composed active mesh
        $\mathcal{T}_{\mathrm{act}}$, with cells of three sizes. Right: its splitting into the per-level meshes
        $\mathcal{T}_0,\mathcal{T}_1,\dots,\mathcal{T}_L$, drawn one above the other in an exploded view. Every active
        cell belongs to one level $\ell$, so the active mesh decomposes into the direct
        sum~\eqref{eq:active_decomposition} of per-level active subspaces; on each level the cells that are refined
        further (inactive there) are shaded gray, while the active cells are white. The refinement edge $E_\ell$ between
        levels $\ell-1$ and $\ell$ is drawn (green) on the finer level $\ell$. The smoother on level $\ell$ acts only on
        the interior DoFs $S_\ell$ of the refined region of that level; the values on $E_\ell$ are interpolated
        coarse-grid data that the smoother must not modify, and the residual crosses $E_\ell$ only through the
        inter-grid transfers.}
    \label{fig:local_mg_levels}
\end{figure}
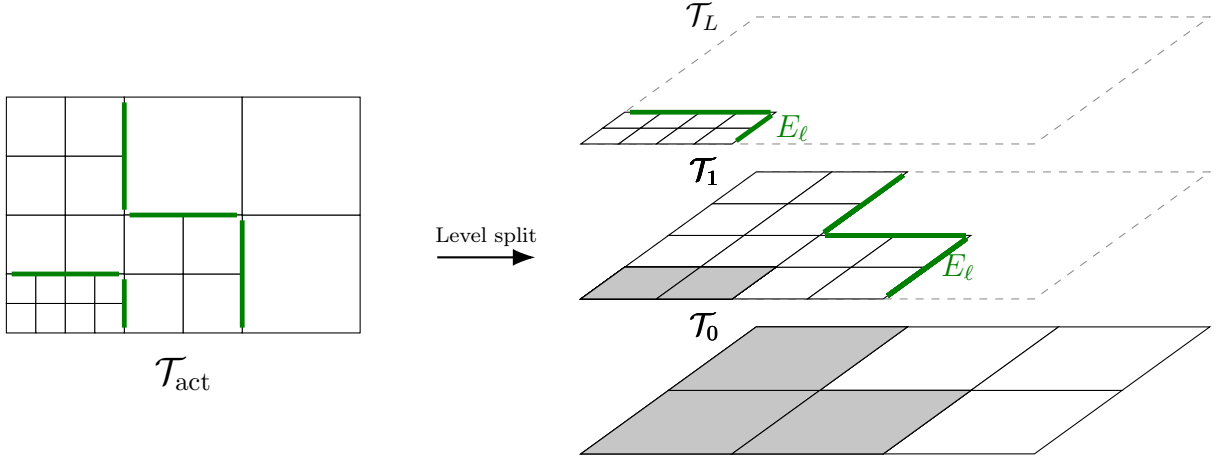

\paragraph{Level-wise spaces and the active-mesh decomposition.}
For each $\ell \in \{0, \dots, L\}$ let $\mathcal{T}_\ell$ denote the collection of \emph{all} cells of level $\ell$,
regardless of whether they are active or further refined, and let $\mathbb{V}_{\text{cell}, \ell}$ be the associated
cell-wise primal space, with dual $\mathbb{V}_{\text{cell}, \ell}^*$. Within $\mathbb{V}_{\text{cell}, \ell}$ we single
out the \emph{active subspace}
\begin{equation}
    \mathbb{V}_{\text{cell}, \ell}^{\text{act}} := \bigl\{ v \in \mathbb{V}_{\text{cell}, \ell} \,:\, v|_K = 0 \text{ for every refined } K \in \mathcal{T}_\ell \bigr\},
    \label{eq:active_subspace}
\end{equation}
i.e., those cell-wise functions whose support is contained in the union of active cells of level $\ell$. Because every active cell has a unique level, the cell-wise active space decomposes as a level-wise direct sum
\begin{equation}
    \mathbb{V}_{\text{cell}, \text{act}} \;=\; \bigoplus_{\ell=0}^L \mathbb{V}_{\text{cell}, \ell}^{\text{act}},
    \label{eq:active_decomposition}
\end{equation}
with the canonical projections $\Pi_\ell : \mathbb{V}_{\text{cell}, \text{act}} \to \mathbb{V}_{\text{cell}, \ell}^{\text{act}}$ and inclusions $\iota_\ell : \mathbb{V}_{\text{cell}, \ell}^{\text{act}} \hookrightarrow \mathbb{V}_{\text{cell}, \ell}$ given by zero extension to refined cells; we denote its left inverse, which restricts a primal vector to its components on active cells, by $\iota_\ell^\dagger$. By duality, $\iota_\ell^* : \mathbb{V}_{\text{cell}, \ell}^* \to (\mathbb{V}_{\text{cell}, \ell}^{\text{act}})^*$ restricts a dual vector to its components on active cells, and its right inverse (the zero extension of a dual vector defined on active cells) will be denoted $\iota_\ell^{*\dagger}$.

The decomposition~\eqref{eq:active_decomposition} furnishes the bridge between the active mesh and the level-wise
multigrid storage. Given a cell-wise dual vector $\tilde{r}_{\text{act}} \in \mathbb{V}_{\text{cell}, \text{act}}^*$,
such as the unassembled residual of the active problem, the components
\begin{equation}
    \tilde{r}_\ell^{\text{act}} := \tilde{r}_{\text{act}}\big|_{\mathbb{V}_{\text{cell}, \ell}^{\text{act}}} \in (\mathbb{V}_{\text{cell}, \ell}^{\text{act}})^*, \qquad \tilde{r}_\ell := \iota_\ell^{*\dagger} \tilde{r}_\ell^{\text{act}} \in \mathbb{V}_{\text{cell}, \ell}^*,
    \label{eq:level_distribution}
\end{equation}
provide the right-hand side seen by the V-cycle on each level: $\tilde{r}_\ell$ vanishes identically on refined cells of $\mathcal{T}_\ell$ and equals $\tilde{r}_{\text{act}}$ on the active cells of that level. Conversely, given a tuple of level-wise primal corrections $(e_\ell)_{\ell=0}^L \in \prod_\ell \mathbb{V}_{\text{cell}, \ell}$, only the active components $\iota_\ell^\dagger e_\ell \in \mathbb{V}_{\text{cell}, \ell}^{\text{act}}$ are physically meaningful, and the resulting correction on the active mesh is
\begin{equation}
    \delta u_{\text{act}} \;=\; \sum_{\ell=0}^{L} \iota_\ell^\dagger e_\ell \;\in\; \mathbb{V}_{\text{cell}, \text{act}},
    \label{eq:level_gather}
\end{equation}
the sum being direct in view of~\eqref{eq:active_decomposition}. Equations~\eqref{eq:level_distribution} and~\eqref{eq:level_gather} connect the active mesh and the multigrid hierarchy. After the residual has been distributed via~\eqref{eq:level_distribution}, the V-cycle operates within the per-level cell-wise spaces $\mathbb{V}_{\text{cell}, \ell}$.

The complete iteration is summarised in Algorithm~\ref{alg:mg_local_iteration}; the recursive V-cycle invoked from
line~\ref{line:mg_local_call} is detailed in Algorithm~\ref{alg:mg_local_vcycle}.

\begin{algorithm}[H]
    \caption{Local adaptive multigrid iteration on $\mathcal{T}_{\text{act}}$.}
    \label{alg:mg_local_iteration}
    \begin{algorithmic}[1]
        \Require Active iterate $u_{\text{act}}$, active right-hand side $b_{\text{act}}$, finest level $L$.
        \Ensure Updated active iterate $u_{\text{act}}$.
        \State $\tilde{r}_{\text{act}} \gets b_{\text{act}} - A_{\text{act}}\, u_{\text{act}}$ \Comment{Unassembled residual on $\mathcal{T}_{\text{act}}$}
        \State $\tilde{r}_\ell \gets \iota_\ell^{*\dagger} (\tilde{r}_{\text{act}}|_{\mathbb{V}_{\text{cell}, \ell}^{\text{act}}})$ \textbf{ for } $\ell = 0, \dots, L$ \Comment{Level-wise distribution~\eqref{eq:level_distribution}}
        \State $e_\ell \gets 0$ \textbf{ for } $\ell = 0, \dots, L$ \Comment{Level-wise correction iterates}
        \State $e_L \gets MG_{\text{local}}(A_L, \tilde{r}_L, e_L, L)$ \label{line:mg_local_call} \Comment{Recursive V-cycle, Algorithm~\ref{alg:mg_local_vcycle}}
        \State $\delta u_{\text{act}} \gets \sum_{\ell=0}^L \iota_\ell^\dagger e_\ell$ \Comment{Level-wise reassembly~\eqref{eq:level_gather}}
        \State $u_{\text{act}} \gets u_{\text{act}} + \delta u_{\text{act}}$
        \State \Return $u_{\text{act}}$
    \end{algorithmic}
\end{algorithm}

\begin{algorithm}[H]
    \caption{Local adaptive multigrid $V$-cycle: $MG_{\text{local}}(A_\ell, b_\ell, u_\ell, \ell)$.}
    \label{alg:mg_local_vcycle}
    \begin{algorithmic}[1]
        \Require Level $\ell$, current iterate $u_\ell$, level right-hand side $b_\ell$ (already containing the contribution from~\eqref{eq:level_distribution}).
        \If{$\ell = 0$}
        \State $u_0 \gets A_0^{-1} b_0$ \Comment{Coarse grid solve}
        \State \Return $u_0$
        \EndIf
        \State $u_\ell \gets \mathcal{S}_\ell^{\nu_1}(A_\ell, b_\ell, u_\ell)$ \Comment{Pre-smoothing with edge constraints}
        \State $r_\ell \gets b_\ell - A_\ell\, u_\ell$ \Comment{Unassembled cell-wise residual}
        \State $b_{\ell-1} \gets b_{\ell-1} + \mathcal{R}_\ell^{\ell-1} r_\ell$ \Comment{Accumulate restricted residual into the level-$\ell-1$ contribution}
        \State $u_{\ell-1} \gets MG_{\text{local}}(A_{\ell-1}, b_{\ell-1}, u_{\ell-1}, \ell - 1)$ \Comment{Recursive coarse grid correction}
        \State $u_\ell \gets u_\ell + \mathcal{P}_{\ell-1}^\ell u_{\ell-1}$ \Comment{Prolongate}
        \State $u_\ell \gets \mathcal{S}_\ell^{\nu_2}(A_\ell, b_\ell, u_\ell)$ \Comment{Post-smoothing with edge constraints}
        \State \Return $u_\ell$
    \end{algorithmic}
\end{algorithm}

In Algorithm~\ref{alg:mg_local_vcycle} the level quantities $b_{\ell-1}$ and $u_{\ell-1}$ refer to the shared per-level
states initialized in Algorithm~\ref{alg:mg_local_iteration}: $b_{\ell-1}$ enters carrying the level-$(\ell-1)$
component $\tilde{r}_{\ell-1}$ of the distribution~\eqref{eq:level_distribution}, and $u_{\ell-1}$ is the level
correction iterate $e_{\ell-1}$, initialized to zero, that the outer iteration later reassembles
via~\eqref{eq:level_gather}. The accumulating update $b_{\ell-1} \gets b_{\ell-1} + \mathcal{R}_\ell^{\ell-1} r_\ell$
on line~7 of Algorithm~\ref{alg:mg_local_vcycle} reflects the additive nature of the
decomposition~\eqref{eq:active_decomposition}: on entry the level-wise right-hand side $b_{\ell-1}$ is already
non-zero, carrying the component of $\tilde{r}_{\text{act}}$ supported on $\mathbb{V}_{\text{cell},
            \ell-1}^{\text{act}}$ that was placed there by~\eqref{eq:level_distribution}, and the restricted fine-grid residual
must be superposed on it rather than overwriting it. Equivalently, the recursion can be read as solving the global
active-grid problem by additive subspace correction along~\eqref{eq:active_decomposition}, with the smoother on level
$\ell$ acting only on the subspace $\mathbb{V}_{\text{cell}, \ell}^{\text{act}}$ that physically lives on that level.

In a classically assembled framework the irregularities at refinement edges would force the introduction of dedicated
edge operators and an explicit hanging-node constraint matrix. The remainder of this subsection establishes the two
cell-wise counterparts that replace them: a refinement-edge mask on the smoother's update (\secref{sec:edge_masking}),
and an algebraic equivalence of the unmasked cell-wise residual restriction with the classical constrained restriction
(\secref{sec:residual_equivalence}).

\subsubsection{Level-Wise Masking and Edge Constraints}
\label{sec:edge_masking}
Within the global space $\mathbb{V}_{\ell}$ defined over $\mathcal{T}_\ell$, the degrees of freedom (DoFs) can be partitioned into two mutually exclusive sets:
\begin{itemize}
    \item \textbf{Subdomain Interior ($S_\ell$):} DoFs belonging to the interior of the level-$\ell$ subdomain $\Omega_\ell := \bigcup_{K \in \mathcal{T}_\ell} K$, the region covered by \emph{all} cells of level $\ell$, whether active or further refined. The smoother acts on the whole of $S_\ell$, not merely on the DoFs of active cells; this matches the local-smoothing formulation of~\cite{JanssenKanschat11}.
    \item \textbf{Refinement Edge ($E_\ell$):} DoFs located on the inter-level boundary $\partial \Omega_\ell \setminus \partial \Omega$, representing hanging nodes that connect level $\ell$ to the coarser level $\ell-1$. DoFs on the physical Dirichlet boundary $\partial \Omega$ are excluded: they are eliminated in the usual way and belong to neither set.
\end{itemize}
This partition induces a block structure for the classically assembled discrete level operator $A_\ell$:
\begin{equation}
    A_\ell = \begin{pmatrix} A^{SS} & A^{SE} \\ A^{ES} & A^{EE} \end{pmatrix}.
\end{equation}

In a classical algebraic setting, the local smoothing procedure relies on explicitly separating these blocks. Following
prolongation, the refinement edge DoFs receive interpolated coarse-grid values, denoted $u^E$. To smooth the interior,
this boundary data is treated as fixed Dirichlet constraints. The inward coupling is explicitly computed and moved to
the right-hand side, forcing the smoother to act exclusively on the interior matrix block aiming to solve $A^{SS} u^S =
    f^S - A^{SE} u^E$. After the interior solution $u^S$ is updated, the residual on the refinement edge must be evaluated
to propagate the fine-grid error back to the coarse level. This requires an explicit edge operator to compute the
outward coupling fluxes $r^E = f^E - A^{ES} u^S - A^{EE} u^E$. This edge residual is then restricted to drive the
subsequent coarse-grid correction.

In the matrix-free cell-wise representation, we evaluate the full local volumetric operator $A_{\text{cell}, \ell}
    \colon \mathbb{V}_{\text{cell}, \ell} \to \mathbb{V}_{\text{cell}, \ell}^*$ in a single pass over the elements in
$\mathcal{T}_\ell$. To recover local smoothing in this cell-wise setting, we introduce the interior-DoF subspace
$\mathbb{V}_{\text{cell}, \ell}^{S_\ell} \subset \mathbb{V}_{\text{cell}, \ell}$ of cell-wise functions vanishing on
the refinement edge $E_\ell$, together with the canonical primal inclusion and its formal adjoint
\begin{equation}
    \iota_{S_\ell} \colon \mathbb{V}_{\text{cell}, \ell}^{S_\ell} \hookrightarrow \mathbb{V}_{\text{cell}, \ell}, \qquad \iota_{S_\ell}^* \colon \mathbb{V}_{\text{cell}, \ell}^* \to \bigl(\mathbb{V}_{\text{cell}, \ell}^{S_\ell}\bigr)^*,
    \label{eq:edge_mask_inclusion}
\end{equation}
in line with the convention used for the active-subspace pair $(\iota_\ell, \iota_\ell^*)$: the primal map $\iota_{S_\ell}$ extends a vector by zero on $E_\ell$ (thereby imposing the homogeneous Dirichlet condition that defines local smoothing), while the dual map $\iota_{S_\ell}^*$ restricts a level dual vector to its components on $S_\ell$. The two maps will be applied on opposite sides of the smoothing preconditioner: $\iota_{S_\ell}^*$ on the dual residual and $\iota_{S_\ell}$ on the primal correction, yielding the local relaxation operator described in \secref{sec:smoother_construction}.

Conceptually, the dual-side restriction $\iota_{S_\ell}^*$ discards the outward fluxes accumulated on $E_\ell$ that are
reserved for the coarse-grid correction, and the primal-side inclusion $\iota_{S_\ell}$ enforces the inter-level
Dirichlet boundary condition by re-inserting zero on those entries. Because the full volumetric operator reads the
non-zero boundary values $u^E$ during the cell integration loop, it evaluates the inward coupling ($A^{SE}$) on the
fly, and the local relaxation acts on the interior ($A^{SS}$). By evaluating the unmasked residual using the full
operator at the end of the smoothing step, the outward fine-to-coarse fluxes ($A^{ES}$) accumulate at the refinement
edges.

\subsubsection{Equivalence of the Unassembled Transfer and Residual Operations}
\label{sec:residual_equivalence}

Having established the topological mapping between the classical block operators and our cell-wise evaluation strategy,
we now demonstrate that the geometric, cell-wise transfer operations enforce the required constraints and correctly
route the residual fluxes. Let us first show that the geometrically prolonged field is continuous across the
inter-level refinement edge.

\begin{lemma}[Continuity of the Prolonged Field]
    \label{lem:prolonged_field_continuity}
    Let $e_{\ell-1} \in \mathbb{V}_{\text{cell}, \ell-1}$ be a continuous coarse-grid field. The geometrically prolonged fine-grid field $e_\ell = \mathcal{P}_{\text{cell}, \ell-1}^\ell e_{\ell-1}$ is $C^0$-continuous across the inter-level refinement edge $E_\ell$.
\end{lemma}

\begin{proof}
    The cell-wise prolongation operator $\mathcal{P}_{\text{cell}, \ell-1}^\ell$ evaluates the polynomial embedding from the coarse finite element space into the fine space. Because the input field $e_{\ell-1}$ is continuous across coarse element interfaces, its geometric interpolation onto the refinement edge $E_\ell$ yields uniquely defined polynomial values. Consequently, the resulting fine-grid degrees of freedom evaluated on $E_\ell$ match the interpolation of the adjacent continuous coarse field, ensuring $C^0$-continuity across the non-conforming interface.
\end{proof}

% Characterize the hanging node constraint operator by its defining property.
In classical algebraically assembled formulations, inter-level continuity at non-conforming interfaces is maintained
through an explicit constraint operator $C_{HN} \colon \mathbb{V}_{CG, \ell} \to \mathbb{V}_{CG, \ell}$. A remark on
its domain is in order: each level mesh $\mathcal{T}_\ell$ is conforming, so within $\mathbb{V}_{CG, \ell}$ itself no
degree of freedom is constrained; the hanging nodes are a property of the \emph{active} mesh, on which the level-$\ell$
DoFs lying on $E_\ell$ are dependent. The classical local multigrid method~\cite{JanssenKanschat11} therefore applies,
on the level space, the edge-constraint operator that acts as the identity on all DoFs in $S_\ell$ and replaces the
values on $E_\ell$ with the exact polynomial interpolation of the supporting coarse degrees of freedom. Rather than
fixing a particular matrix representation, we characterize $C_{HN}$ by the single property that our equivalence proof
requires: for any continuous coarse field $v_{CG, \ell-1} \in \mathbb{V}_{CG, \ell-1}$,
\begin{equation}
    C_{HN} \mathcal{P}_{CG, \ell-1}^\ell v_{CG, \ell-1} = \mathcal{P}_{CG, \ell-1}^\ell v_{CG, \ell-1}.
    \label{eq:constraints_prolongation}
\end{equation}
Property~\eqref{eq:constraints_prolongation} states that prolonged coarse fields are invariant under the constraints,
and it holds for the edge-constraint operator of~\cite{JanssenKanschat11} because, by
Lemma~\ref{lem:prolonged_field_continuity}, the values that the geometric prolongation places on $E_\ell$ \emph{are}
the exact interpolation of the supporting coarse DoFs, leaving no constraints to enforce. Everything below
uses only~\eqref{eq:constraints_prolongation}, so the results apply verbatim to any classical formulation whose
constraint operator satisfies it.

Transferring the residual to coarser grid requires computing the outward coupling fluxes ($r^E = f^E - A^{ES} u^S -
    A^{EE} u^E$) and applying the transposed constraint matrix $C_{HN}^T \colon \mathbb{V}_{CG, \ell}^* \to \mathbb{V}_{CG,
        \ell}^*$ to accumulate these edge fluxes onto the supporting coarse-grid DoFs. In our cell-wise approach, we bypass
both the explicit edge operator evaluation and the constraint matrix.

\begin{lemma}[Equivalence of the Residual Transfer]
    \label{lem:constrained_residual}
    Let $C_{HN}$ satisfy~\eqref{eq:constraints_prolongation}. Then the operator identity
    \begin{equation}
        \mathcal{G}_{\ell-1}^T \mathcal{R}_{\text{cell}, \ell}^{\ell-1} = \mathcal{R}_{CG, \ell}^{\ell-1} C_{HN}^T \mathcal{G}_\ell^T
    \end{equation}
    holds on all of $\mathbb{V}_{\text{cell}, \ell}^*$. In particular, for the unmasked cell-wise residual $\tilde{r}_\ell = \tilde{b}_\ell - A_{\text{cell}, \ell} u_\ell$ evaluated after local smoothing, the global assembly of its cell-wise restriction equals the result of assembling the fine-grid residual, applying the transposed hanging-node constraint operator $C_{HN}^T$, and applying the classical global restriction $\mathcal{R}_{CG, \ell}^{\ell-1}$.
\end{lemma}

\begin{proof}
    Let $\tilde{r}_\ell \in \mathbb{V}_{\text{cell}, \ell}^*$ be arbitrary. To prove the identity, we evaluate the action of both sides on an arbitrary continuous coarse-grid test vector $v_{CG, \ell-1} \in \mathbb{V}_{CG, \ell-1}$ using duality pairings; since both sides are elements of $\mathbb{V}_{CG, \ell-1}^*$, agreement of all such pairings determines them completely.

    First, we consider the classical constrained restriction (the right-hand side). By sequentially applying the
    definitions of the adjoint operators (the global restriction $\mathcal{R}_{CG, \ell}^{\ell-1} = (\mathcal{P}_{CG,
            \ell-1}^\ell)^T$, the transposed constraint matrix $C_{HN}^T$, and the scatter operator $\mathcal{G}_\ell^T$), we
    transfer the operations from the dual space to the primal space:
    \begin{align*}
        \langle \mathcal{R}_{CG, \ell}^{\ell-1} C_{HN}^T \mathcal{G}_\ell^T \tilde{r}_\ell, v_{CG, \ell-1} \rangle_{CG, \ell-1}
         & = \langle C_{HN}^T \mathcal{G}_\ell^T \tilde{r}_\ell, \mathcal{P}_{CG, \ell-1}^\ell v_{CG, \ell-1} \rangle_{CG, \ell} \\
         & = \langle \mathcal{G}_\ell^T \tilde{r}_\ell, C_{HN} \mathcal{P}_{CG, \ell-1}^\ell v_{CG, \ell-1} \rangle_{CG, \ell}   \\
         & = \langle \tilde{r}_\ell, \mathcal{G}_\ell \mathcal{P}_{CG, \ell-1}^\ell v_{CG, \ell-1} \rangle_{\text{cell}, \ell},
    \end{align*}
    where we have explicitly utilized the identity from Eq.~\eqref{eq:constraints_prolongation} ($C_{HN} \mathcal{P}_{CG, \ell-1}^\ell = \mathcal{P}_{CG, \ell-1}^\ell$) in the final step. The resulting primal vector in the right argument, $\mathcal{G}_\ell \mathcal{P}_{CG, \ell-1}^\ell v_{CG, \ell-1}$, represents the coarse test vector prolonged to the global fine space, implicitly satisfying the constraints, and finally gathered into the cell-wise representation.

    Next, we evaluate the cell-wise restriction (the left-hand side). Applying the adjoint definitions for the coarse
    scatter operator $\mathcal{G}_{\ell-1}^T$ and the cell-wise restriction $\mathcal{R}_{\text{cell}, \ell}^{\ell-1} =
        (\mathcal{P}_{\text{cell}, \ell-1}^\ell)^T$ yields:
    \begin{align*}
        \langle \mathcal{G}_{\ell-1}^T \mathcal{R}_{\text{cell}, \ell}^{\ell-1} \tilde{r}_\ell, v_{CG, \ell-1} \rangle_{CG, \ell-1}
         & = \langle \mathcal{R}_{\text{cell}, \ell}^{\ell-1} \tilde{r}_\ell, \mathcal{G}_{\ell-1} v_{CG, \ell-1} \rangle_{\text{cell}, \ell-1} \\
         & = \langle \tilde{r}_\ell, \mathcal{P}_{\text{cell}, \ell-1}^\ell \mathcal{G}_{\ell-1} v_{CG, \ell-1} \rangle_{\text{cell}, \ell}.
    \end{align*}
    The primal vector here, $\mathcal{P}_{\text{cell}, \ell-1}^\ell \mathcal{G}_{\ell-1} v_{CG, \ell-1}$, is obtained by gathering the coarse vector into the cell-wise format and applying the geometric tensor-product prolongation.

    Because the operations of gathering and prolonging commute (Lemma~\ref{lem:prolongation_commutativity}), the two
    cell-wise primal test vectors are identical:
    \begin{equation*}
        \mathcal{P}_{\text{cell}, \ell-1}^\ell \mathcal{G}_{\ell-1} v_{CG, \ell-1} = \mathcal{G}_\ell \mathcal{P}_{CG, \ell-1}^\ell v_{CG, \ell-1}.
    \end{equation*}
    Since the duality pairings are equal for any arbitrary test vector $v_{CG, \ell-1}$, the operator identity holds. This proves that the local cell-wise restriction of the unmasked residual replicates the classical explicitly constrained residual transfer.
\end{proof}

Note that we have not formed the constraint matrix $C_{HN}$ explicitly; we needed only the property of the constraint
matrix and its connection with the geometric prolongation, Equation~(\ref{eq:constraints_prolongation}).

\subsubsection{Equivalence with Classical Local Multigrid}
\label{sec:mg_local_equivalence}

We now combine the two lemmas above to identify Algorithm~\ref{alg:mg_local_vcycle} with the classical local multigrid
V-cycle on conforming finite element spaces studied in~\cite{JanssenKanschat11}. The only ingredient that is not
already covered by the inter-grid transfer equivalences is the smoother, whose action must agree with that of a
classical local smoother once we identify cell-wise iterates with their continuous representatives via the gather
operator $\mathcal{G}_\ell$. We make this precise as a single assumption.

\begin{assumption}[Smoother equivalence]
    \label{ass:smoother_equivalence}
    There exists a classical local level smoother $\mathcal{S}_\ell^{CG}$, mapping a triple of level operator, assembled right-hand side and iterate $(A_{CG,\ell},\, b_{CG,\ell},\, u_{CG,\ell})$ to an updated iterate in $\mathbb{V}_{CG,\ell}$,
    in the sense of~\cite{JanssenKanschat11} (a relaxation that acts on the interior degrees of freedom $S_\ell$ and leaves the refinement-edge values $E_\ell$ untouched) such that, for every continuous iterate $u_{CG} \in \mathbb{V}_{CG,\ell}$ and every unassembled dual vector $\tilde{b}_\ell \in \mathbb{V}_{\text{cell},\ell}^*$,
    \begin{equation}
        \mathcal{S}_\ell\bigl(A_\ell,\, \tilde{b}_\ell,\, \mathcal{G}_\ell u_{CG}\bigr) \;=\; \mathcal{G}_\ell\, \mathcal{S}_\ell^{CG}\bigl(A_{CG,\ell},\, \mathcal{G}_\ell^T \tilde{b}_\ell,\, u_{CG}\bigr).
        \label{eq:smoother_equivalence}
    \end{equation}
\end{assumption}

In words, applied to a globally continuous iterate the cell-wise smoother $\mathcal{S}_\ell$ produces an output that
lies in the range of $\mathcal{G}_\ell$ and gathers to the classical local smoother's output. The relaxation built from
a masked Jacobi preconditioner in the next subsection is shown there to satisfy~\eqref{eq:smoother_equivalence}.

\begin{theorem}[Equivalence of cell-wise and classical local multigrid]
    \label{thm:mg_local_equivalence}
    Let Assumption~\ref{ass:smoother_equivalence} hold, and let $\mathrm{MG}_{\text{local}}^{CG}$ denote the classical local V-cycle of~\cite{JanssenKanschat11} built from the operators $A_{CG,\ell}$, the smoother $\mathcal{S}_\ell^{CG}$, the prolongation $\mathcal{P}_{CG,\ell-1}^\ell$, the restriction $\mathcal{R}_{CG,\ell}^{\ell-1}$ and a hanging-node constraint operator $C_{HN}$ satisfying~\eqref{eq:constraints_prolongation}. For every level $\ell$, every continuous iterate $u_{CG,\ell} \in \mathbb{V}_{CG,\ell}$ and every unassembled right-hand side $\tilde{b}_\ell \in \mathbb{V}_{\text{cell},\ell}^*$, the cell-wise V-cycle of Algorithm~\ref{alg:mg_local_vcycle} satisfies
    \begin{equation}
        \mathrm{MG}_{\text{local}}\bigl(A_\ell,\, \tilde{b}_\ell,\, \mathcal{G}_\ell u_{CG,\ell},\, \ell\bigr) \;=\; \mathcal{G}_\ell\, \mathrm{MG}_{\text{local}}^{CG}\bigl(A_{CG,\ell},\, \mathcal{G}_\ell^T \tilde{b}_\ell,\, u_{CG,\ell},\, \ell\bigr).
        \label{eq:mg_equivalence}
    \end{equation}
    Consequently, the iteration matrix of the cell-wise V-cycle, restricted to the gathered subspace $\mathcal{G}_\ell \mathbb{V}_{CG,\ell} \subset \mathbb{V}_{\text{cell},\ell}$, is similar to that of $\mathrm{MG}_{\text{local}}^{CG}$: any convergence estimate that holds for the classical local V-cycle built from $\mathcal{S}_\ell^{CG}$ transfers verbatim to Algorithm~\ref{alg:mg_local_vcycle}.
\end{theorem}

\begin{proof}
    Induction on $\ell$. The base case $\ell = 0$ is immediate: the coarse-level solve $u_0 \gets A_0^{-1} b_0$ acts on the conforming coarsest space, where $\mathcal{G}_0$ is a bijection and the cell-wise and classical operators coincide.

    Assume~\eqref{eq:mg_equivalence} holds at level $\ell-1$. Lemma~\ref{lem:prolonged_field_continuity} states that the
    cell-wise prolongation produces $C^0$-continuous fields, hence $\mathcal{P}_{\text{cell},\ell-1}^\ell$ maps
    $\mathcal{G}_{\ell-1} \mathbb{V}_{CG,\ell-1}$ into $\mathcal{G}_\ell \mathbb{V}_{CG,\ell}$, and
    Lemma~\ref{lem:prolongation_commutativity} gives the intertwining $\mathcal{P}_{\text{cell},\ell-1}^\ell
        \mathcal{G}_{\ell-1} = \mathcal{G}_\ell \mathcal{P}_{CG,\ell-1}^\ell$. Lemma~\ref{lem:constrained_residual} provides
    the dual identity $\mathcal{G}_{\ell-1}^T \mathcal{R}_{\text{cell},\ell}^{\ell-1} = \mathcal{R}_{CG,\ell}^{\ell-1}
        C_{HN}^T \mathcal{G}_\ell^T$, valid on all of $\mathbb{V}_{\text{cell},\ell}^*$ and in particular on the unassembled
    residual evaluated after smoothing. Finally, Assumption~\ref{ass:smoother_equivalence} provides the matching identity
    for the pre- and post-smoothing steps. Substituting these three identities line by line into
    Algorithm~\ref{alg:mg_local_vcycle} reproduces, after gathering, the corresponding lines of the classical local V-cycle
    of~\cite{JanssenKanschat11} applied to $(A_{CG,\ell}, \mathcal{G}_\ell^T \tilde{b}_\ell, u_{CG,\ell})$, completing the
    induction.

    Convergence estimates for the classical method are stated as norm bounds on the iteration matrix of
    $\mathrm{MG}_{\text{local}}^{CG}$ acting on $\mathbb{V}_{CG,\ell}$; by~\eqref{eq:mg_equivalence} the cell-wise
    iteration matrix on $\mathcal{G}_\ell \mathbb{V}_{CG,\ell}$ is conjugate to it via $\mathcal{G}_\ell$ and therefore
    inherits the same spectral bounds.
\end{proof}

\begin{corollary}[Equivalence of the full iteration]
    \label{cor:full_iteration_equivalence}
    Under the hypotheses of Theorem~\ref{thm:mg_local_equivalence}, one step of Algorithm~\ref{alg:mg_local_iteration}
    applied to a continuous active iterate coincides, after gathering, with one step of the classical local multigrid
    iteration of~\cite{JanssenKanschat11} on the assembled active-mesh system. In particular, the reassembled correction
    $\delta u_{\text{act}}$ of~\eqref{eq:level_gather} satisfies the hanging-node constraints of the active mesh: on
    every refinement edge the fine-side values are those placed by the prolongation of the coarse-level correction,
    which by Lemma~\ref{lem:prolonged_field_continuity} interpolate the coarse field, while the level smoothers
    leave $E_\ell$ untouched by construction.
\end{corollary}

\begin{proof}
    The level-wise distribution~\eqref{eq:level_distribution} assigns to each level the components of the unassembled
    active residual supported on its active cells; assembling level by level, this is the classical splitting
    of the assembled active residual into the level right-hand sides (the \texttt{copy\_to\_mg} operation
    of~\cite{JanssenKanschat11}). Theorem~\ref{thm:mg_local_equivalence} identifies the recursive V-cycle with its
    classical counterpart. For the reassembly~\eqref{eq:level_gather}, note that on each level only the active
    components $\iota_\ell^\dagger e_\ell$ enter the sum, the active subspaces being disjoint
    by~\eqref{eq:active_decomposition}; the values that the prolongation deposited on cells refined further are
    discarded on level $\ell$ and recovered, refined, from the finer levels, so no contribution is counted twice. The
    continuity claim follows since each $e_\ell$ is continuous on $\mathcal{T}_\ell$ (smoother output and prolonged
    fields are continuous), the smoothers never modify $E_\ell$, and the $E_\ell$ values therefore remain the
    prolongation of the coarse correction.
\end{proof}

The cell-wise V-cycle of Algorithm~\ref{alg:mg_local_vcycle} is therefore a faithful matrix-free realisation of the
classical local multigrid method whenever the smoother satisfies Assumption~\ref{ass:smoother_equivalence}. Two remarks
delimit what the equivalence does and does not provide. First, it reduces the convergence analysis of the cell-wise
solver to a classical question: Algorithm~\ref{alg:mg_local_vcycle} converges as fast as the classical local V-cycle
built from $\mathcal{S}_\ell^{CG}$, so the estimates of~\cite{JanssenKanschat11} apply whenever $\mathcal{S}_\ell^{CG}$
satisfies the smoother hypotheses of that theory. For the damped Jacobi relaxation constructed below, this is the
standard requirement $\omega \lesssim 1/\lambda_{\max}(D_{CG,\ell}^{-1} A_{CG,\ell})$, yielding level-independent
(though not $p$-independent) convergence. Second, the equivalence framework is not tied to Jacobi: any classical local
smoother whose cell-wise realisation satisfies Assumption~\ref{ass:smoother_equivalence} (such as the overlapping
vertex-patch smoothers~\cite{pavarino1993additive,WitteArndtKanschat21,brubeck2021scalable}, whose convergence improves
with the polynomial degree) inherits its convergence theory through Theorem~\ref{thm:mg_local_equivalence}, since such
smoothers consume the residual through its assembly and return continuous corrections (they are of the admissible
form~\eqref{eq:precond_consistency}). The remainder of this section constructs the simplest such smoother, the masked
damped Jacobi relaxation.

\subsection{Construction of the Smoother}
\label{sec:smoother_construction}

It remains to exhibit a concrete cell-wise smoother $\mathcal{S}_\ell$ that satisfies the equivalence
Assumption~\ref{ass:smoother_equivalence}, thereby completing the hypotheses of Theorem~\ref{thm:mg_local_equivalence}.
We adopt a damped Jacobi relaxation built from the trivial continuity-restoring preconditioner (realized as direct
stiffness summation in \secref{sec:smoother_realization} below) and composed, on its dual and primal sides, with the
inclusion pair $(\iota_{S_\ell}, \iota_{S_\ell}^*)$ of \secref{sec:edge_masking}. We use the symbol $c_\ell \in
    \mathbb{V}_{\text{cell}, \ell}$ for the primal correction (reserving $\Delta$ for the Laplacian, which appears in the
numerical experiments).

\paragraph{Damped Jacobi correction on the interior subspace.}
Let $S \colon \mathbb{V}_{\text{cell}, \ell}^* \to \mathcal{G}_\ell \mathbb{V}_{CG, \ell} \subset
    \mathbb{V}_{\text{cell}, \ell}$ denote an operator that assembles a cell-wise dual vector and re-injects the resulting
$C^0$-continuous primal field into the redundant cell-wise storage. This is the \emph{direct stiffness summation} (DSS)
operator, whose role as the minimal continuity-restoring preconditioner and whose algorithmic realization are the
subject of \secref{sec:smoother_realization}. Together with the cell-wise diagonal $D_{\text{cell}, \ell} =
    \operatorname{diag}(\mathcal{G}_\ell\, d_\ell)$, where $d_\ell \in \mathbb{V}_{CG, \ell}$ collects the diagonal entries
of $A_{CG, \ell}$, it gives the cell-wise damped Jacobi preconditioner
\begin{equation}
    P_\ell \;:=\; D_{\text{cell}, \ell}^{-1}\, S \;\colon\; \mathbb{V}_{\text{cell}, \ell}^* \to \mathbb{V}_{\text{cell}, \ell}.
    \label{eq:jacobi_preconditioner}
\end{equation}
Because both $S$ and $D_{\text{cell}, \ell}$ act independently on each DoF entry, the natural restriction of $P_\ell$ to the interior subspace,
\begin{equation}
    P_\ell^{S_\ell} \;\colon\; \bigl(\mathbb{V}_{\text{cell}, \ell}^{S_\ell}\bigr)^* \to \mathbb{V}_{\text{cell}, \ell}^{S_\ell},
    \label{eq:jacobi_interior}
\end{equation}
is well-defined and identifies with $P_\ell$ when restricted to inputs and outputs supported on $S_\ell$.

\paragraph{Smoother definition.}
The level-wise smoother is obtained by sandwiching $P_\ell^{S_\ell}$ between the dual restriction $\iota_{S_\ell}^*$ on
the residual side and the primal inclusion $\iota_{S_\ell}$ on the correction side:
\begin{equation}
    \mathcal{S}_\ell(A_\ell, \tilde{b}_\ell, u_\ell) \;:=\; u_\ell \;+\; \omega\, \iota_{S_\ell}\, P_\ell^{S_\ell}\, \iota_{S_\ell}^* \bigl(\tilde{b}_\ell - A_\ell u_\ell\bigr),
    \label{eq:cellwise_smoother}
\end{equation}
with relaxation parameter $\omega > 0$. The dual restriction $\iota_{S_\ell}^*$ discards the outward fluxes on $E_\ell$ that are reserved for the coarse-grid correction; the interior preconditioner $P_\ell^{S_\ell}$ produces a primal correction supported on $S_\ell$; and the inclusion $\iota_{S_\ell}$ extends this correction by zero on $E_\ell$, enforcing the homogeneous Dirichlet boundary condition that defines local smoothing. Because $S$ produces a $C^0$-continuous primal vector and $D_{\text{cell}, \ell}$ is diagonal in the cell-wise basis, $P_\ell^{S_\ell} \iota_{S_\ell}^* \tilde{r}_\ell$ is already $C^0$-continuous in $\mathbb{V}_{\text{cell}, \ell}^{S_\ell}$ and the inclusion $\iota_{S_\ell}$ yields a vector in $\mathcal{G}_\ell \mathbb{V}_{CG, \ell}$ directly.

\paragraph{Smoother equivalence.}
The construction~\eqref{eq:cellwise_smoother} mirrors, in cell-wise storage, the classical local damped Jacobi smoother
of~\cite{JanssenKanschat11}. Concretely:

\begin{lemma}[Verification of Assumption~\ref{ass:smoother_equivalence}]
    \label{lem:smoother_equivalence}
    Let
    $$\mathcal{S}_\ell^{CG} \colon (A_{CG, \ell}, b_{CG, \ell}, u_{CG, \ell}) \mapsto u_{CG, \ell} + \omega\, \iota^{CG}_{S_\ell} (D^{S_\ell}_{CG, \ell})^{-1} (\iota^{CG}_{S_\ell})^*\bigl(b_{CG, \ell} - A_{CG, \ell} u_{CG, \ell}\bigr)$$
    denote the classical local damped Jacobi smoother on $\mathbb{V}_{CG, \ell}$, where $\iota^{CG}_{S_\ell}$ is the inclusion of the interior-DoF subspace of $\mathbb{V}_{CG, \ell}$ (so that $E_\ell$ entries are frozen). The smoother $\mathcal{S}_\ell$ defined by~\eqref{eq:cellwise_smoother} satisfies the equivalence~\eqref{eq:smoother_equivalence}.
\end{lemma}

\begin{proof}
    Set $u_\ell = \mathcal{G}_\ell u_{CG, \ell}$ and let $\tilde{b}_\ell \in \mathbb{V}_{\text{cell}, \ell}^*$ be arbitrary, with assembled counterpart $b_{CG, \ell} = \mathcal{G}_\ell^T \tilde{b}_\ell$. The dual residual gathers as $\mathcal{G}_\ell^T(\tilde{b}_\ell - A_\ell u_\ell) = b_{CG, \ell} - A_{CG, \ell} u_{CG, \ell}$. The classification of a DoF as interior or edge is a property of the global DoF, shared by all of its cell-local copies, so the dual mask commutes with assembly: zeroing the $E_\ell$ entries of every cell-local copy and then assembling gives the same result as assembling first and zeroing the $E_\ell$ entries of the global vector, i.e.\ the inclusion pair $(\iota_{S_\ell}, \iota_{S_\ell}^*)$ intertwines with the gather--scatter pair $(\mathcal{G}_\ell, \mathcal{G}_\ell^T)$ through the corresponding $CG$ inclusions. Moreover, applying DSS to an $E_\ell$-masked dual vector produces a continuous primal vector that vanishes on $E_\ell$, so the subsequent primal inclusion $\iota_{S_\ell}$ acts as the identity. Combining these observations with the trivial-preconditioner identity $S = \mathcal{G}_\ell \mathcal{I}_{\text{Riesz}} \mathcal{G}_\ell^T$ and the diagonal intertwining $\mathcal{G}_\ell D_{CG, \ell}^{-1} = D_{\text{cell}, \ell}^{-1} \mathcal{G}_\ell$ of~\eqref{eq:jacobi_recap} yields
    $$ \iota_{S_\ell}\, P_\ell^{S_\ell}\, \iota_{S_\ell}^* (\tilde{b}_\ell - A_\ell u_\ell) \;=\; \mathcal{G}_\ell\, \iota^{CG}_{S_\ell} (D^{S_\ell}_{CG, \ell})^{-1} (\iota^{CG}_{S_\ell})^* (b_{CG, \ell} - A_{CG, \ell} u_{CG, \ell}). $$
    Adding to $u_\ell = \mathcal{G}_\ell u_{CG, \ell}$ gives the right-hand side of~\eqref{eq:smoother_equivalence}.
\end{proof}

Together with Theorem~\ref{thm:mg_local_equivalence}, Lemma~\ref{lem:smoother_equivalence} closes the convergence
argument: the local cell-wise V-cycle of Algorithm~\ref{alg:mg_local_vcycle} equipped with the
smoother~\eqref{eq:cellwise_smoother} inherits the convergence estimates of~\cite{JanssenKanschat11}.

\subsection{Convergence Measure}
\label{sec:energy_measure}

Throughout this paper, convergence is monitored in the \emph{energy} pairing
\begin{equation}
    \eta(u) \;=\; \langle \tilde{r}, u \rangle \;=\; \langle \tilde{b} - A u,\, u \rangle,
    \label{eq:energy_measure}
\end{equation}
which pairs the raw, unassembled dual residual with the continuous primal iterate. By the pairing
identity~\eqref{eq:pairing_global} this equals the assembled pairing $\langle b_{CG} - A_{CG} u_{CG}, u_{CG} \rangle$,
so it is computable entirely from local data, without any DSS pass. This is the natural convergence measure of the
framework: the Euclidean norm of the \emph{raw} residual is not meaningful (its duplicated entries hold arbitrary
partial sums), and the Euclidean norm of the assembled residual requires an additional DSS application solely for
monitoring purposes. The energy pairing requires neither, and it is also the quantity the framework's solvers natively
control: the equivalence theorem of the companion paper~\cite{wichrowski2026DSS} shows that the communication-free
conjugate gradient iteration identically minimizes the exact global energy functional $\mathcal{J}(u) = \tfrac{1}{2}
    \langle A u, u \rangle - \langle \tilde{b}, u \rangle$, whose gradient pairing with the iterate
is~\eqref{eq:energy_measure}. Writing $e = u^* - u$ for the discrete error, $\eta(u) = a(e, u)$ vanishes at
the discrete solution, and in our experiments its per-cycle reduction tracks that of the energy norm $\|e\|_A$.

\section{Algorithmic Realization of the Smoother}
\label{sec:smoother_realization}

The multigrid theory of the previous section is formulated against an abstract requirement: the smoother must consume
the residual only through its assembly and return a continuity-respecting correction
(cf.~\eqref{eq:precond_consistency}). This section supplies the concrete realization. We first identify the simplest
operator with this property (direct stiffness summation) and assemble the smoothing sweep around it; we then summarize
the structured DSS kernel of the companion paper~\cite{wichrowski2026DSS} and show that the refinement-edge masking
makes it applicable, unchanged, on adaptively refined meshes.

\subsection{Direct Stiffness Summation: the Trivial Continuity-Restoring Preconditioner}
\label{sec:trivial_preconditioner}

% Paragraph: S as the simplest operator of the consistent form.
The simplest operator of the consistent form~\eqref{eq:precond_consistency} is obtained by taking $P_{CG}$ to be the
identity (the Euclidean Riesz isomorphism $\mathcal{I}_{\text{Riesz}} \colon \mathbb{V}_{CG}^* \to \mathbb{V}_{CG}$).
The resulting operator
\begin{equation} \label{eq:dss_operator}
    S = \mathcal{G} \mathcal{I}_{\text{Riesz}} \mathcal{G}^T \colon \mathbb{V}_{\text{cell}}^* \to \mathbb{V}_{\text{cell}}
\end{equation}
is the classical \emph{direct stiffness summation}: it sums the duplicated partial integrals of a dual vector across
all sharing cells and re-injects the accumulated values into every copy, producing a $C^0$-continuous primal vector.
By construction $\text{Im}(S) = \mathcal{G}(\mathbb{V}_{CG})$. $S$ is one realization, not a
requirement: any operator satisfying~\eqref{eq:precond_consistency}, such as an overlapping patch
smoother~\cite{pavarino1993additive,WitteArndtKanschat21,brubeck2021scalable}, could take its place in the smoother
without affecting the equivalence theory of \secref{sec:multigrid}.

% Paragraph: Jacobi scaling commutes with the gather.
Diagonal scaling composes with $S$ at no extra communication cost. Let $d \in \mathbb{V}_{CG}$ collect the diagonal
entries of $A_{CG}$ and let $D_{\text{cell}} = \operatorname{diag}(\mathcal{G} d)$ be the cell-wise diagonal holding
the gathered copies. Since $\mathcal{G}$ merely replicates entries, the scaling intertwines with the gather,
$\mathcal{G} D_{CG}^{-1} = D_{\text{cell}}^{-1} \mathcal{G}$, so that
\begin{equation} \label{eq:jacobi_recap}
    \mathcal{G} D_{CG}^{-1} \mathcal{I}_{\text{Riesz}} \mathcal{G}^T = D_{\text{cell}}^{-1} S,
\end{equation}
i.e.\ the damped Jacobi preconditioner $P_\ell = D_{\text{cell}, \ell}^{-1} S$ of~\eqref{eq:jacobi_preconditioner} is
evaluated as one DSS sweep followed by a pointwise scaling of the local vectors.

\subsection{The Smoothing Sweep}
\label{sec:smoother_sweep}
Algorithm~\ref{alg:smoother} summarizes the practical realization of the smoother~\eqref{eq:cellwise_smoother},
faithfully following the dual-then-primal flow of the formula. Each smoothing sweep evaluates the unassembled
volumetric flux $v_\ell := A_\ell u_\ell \in \mathbb{V}_{\text{cell}, \ell}^*$, forms the unassembled dual residual
$\tilde{r}_\ell = \tilde{b}_\ell - v_\ell$, applies the dual restriction $\iota_{S_\ell}^*$ to zero its $E_\ell$
entries, assembles via DSS to obtain $\tilde{z}_\ell := S \tilde{r}_\ell \in \mathcal{G}_\ell \mathbb{V}_{CG, \ell}$,
and applies the pointwise diagonal scaling $c_\ell = \omega D_{\text{cell}, \ell}^{-1} \tilde{z}_\ell$. Because the
dual residual was already zero on $E_\ell$ before DSS, the resulting primal correction is automatically zero on
$E_\ell$ and $C^0$-continuous, so the inclusion $\iota_{S_\ell}$ acts as the identity and is absorbed silently into the
storage layout.

\begin{algorithm}[H]
    \caption{Level smoother $\mathcal{S}_\ell$ realizing~\eqref{eq:cellwise_smoother}.}
    \label{alg:smoother}
    \begin{algorithmic}[1]
        \Require Level $\ell$, iterate $u_\ell \in \mathbb{V}_{\text{cell}, \ell}$, unassembled right-hand side $\tilde{b}_\ell \in \mathbb{V}_{\text{cell}, \ell}^*$, relaxation $\omega$, inverse cell-wise diagonal $D_{\text{cell}, \ell}^{-1}$.
        \Ensure Updated iterate $u_\ell$.
        \State $v_\ell \gets A_\ell\, u_\ell$ \Comment{Cell-wise volumetric operator}
        \State $\tilde{r}_\ell \gets \tilde{b}_\ell - v_\ell$ \Comment{Unassembled dual residual}
        \State $\tilde{r}_\ell \gets \iota_{S_\ell}^*\, \tilde{r}_\ell$ \Comment{Dual mask: zero $E_\ell$ entries}
        \State $\bar{z}_\ell \gets S\, \tilde{r}_\ell$ \Comment{Inter-cell communication}
        \State $c_\ell \gets \omega\, D_{\text{cell}, \ell}^{-1}\, \bar{z}_\ell$ \Comment{Diagonal Jacobi scaling}
        \State $u_\ell \gets u_\ell + c_\ell$
        \State \Return $u_\ell$
    \end{algorithmic}
\end{algorithm}

A direct inspection of Algorithm~\ref{alg:smoother} together with
Algorithms~\ref{alg:mg_local_iteration}--\ref{alg:mg_local_vcycle} reveals the structural core of the framework: every
operation in one V-cycle is cell-local, except for the single DSS application $S$ on line~4 of the smoother. The
level-wise distribution~\eqref{eq:level_distribution} and reassembly~\eqref{eq:level_gather} read and write each active
cell independently of all others; the unassembled residual evaluation $\tilde{b}_\ell - A_\ell u_\ell$ is a sum of
cell-local terms; the cell-wise restriction $\mathcal{R}_\ell^{\ell-1}$ and prolongation $\mathcal{P}_{\ell-1}^\ell$
are tensor-product Kronecker contractions between a parent and its own children
(Algorithms~\ref{alg:prolongation}--\ref{alg:restriction}); the dual mask $\iota_{S_\ell}^*$ is diagonal in the
cell-wise basis; and the diagonal scaling $D_{\text{cell}, \ell}^{-1}$ is pointwise. The only operator that reads or
writes data shared between distinct cells is $S$, which sums redundantly stored entries across element interfaces, and
it is confined to the smoother. The classical multigrid pipeline is replaced by a pipeline whose inter-cell
communication is deferred to the relaxation step. The outer Krylov accelerator and the inter-grid transfers can then
operate on disjoint memory tiles, and the engineering effort needed to optimize a high-throughput GPU implementation
concentrates onto a single primitive: the DSS step inside the smoother. The DSS algorithm summarized next is therefore
the only piece of inter-cell topology that the solver ever sees.

\subsection{Realization of DSS on Block-Structured Grids}
\label{sec:dss}

% Paragraph: design goal: no indirect addressing, no atomics; layered construction.
The companion paper~\cite{wichrowski2026DSS} develops GPU algorithms that evaluate $S$ without indirect addressing or
atomic operations, exploiting the block-structured cell-wise layout instead of traversing global index lists. Rather
than routing data through the indirect indices of an assembled vector, partial sums are exchanged directly across cell
interfaces via symmetric point-to-point transmissions. We recall the construction here in the detail needed to analyse
its interaction with adaptive refinement; the multigrid theory above requires only the algebraic
form~\eqref{eq:dss_operator}, so any gather-scatter realization~\cite{gslib,libceed2021} could be substituted, but the
interplay between the \emph{structured} kernel below and the refinement-edge mask is what makes the adaptive V-cycle
fast in practice.

% Paragraph: the dimensionally-split structured cascade.
\paragraph{The structured cascade.}
Within a structured macro-block of cells, DSS is realized as a dimensionally-split cascade of $d$ axis-wise passes
(Algorithm~\ref{alg:struct_dss}). A single pass along axis $\alpha$ visits every interior interface normal to $\alpha$,
sums the two coincident copies of each interface DoF, and writes the sum back to both sides, so that after the pass
both copies hold the assembled value of that face pair. Because the block layout is uniform, the exchange partners sit
at statically known offsets: the kernels read and write coalesced memory with hard-coded face offsets, and since every
interface pair is handled by one thread, no atomic operations are needed. The key property is that vertex and edge DoFs
are assembled as a byproduct: in 2D, a vertex shared by four cells receives the sums of its vertical pairs during the
$y$-pass, and the subsequent $x$-pass exchanges these pairwise-summed values across the vertical interfaces, so all
four copies hold the total. The same argument applies recursively in 3D for the four-way edge and eight-way vertex
sums. Sequentiality of the passes is essential, but only $d$ kernel launches are required in total. The kernel sweeps a
block of memory without consulting mesh connectivity.

\begin{algorithm}[H]
    \caption{Dimensionally-split structured summation~\cite{wichrowski2026DSS}.}
    \label{alg:struct_dss}
    \begin{algorithmic}[1]
        \Require Cell-wise dual vector $u \in \mathbb{V}_{\text{cell}}^*$ on a structured block.
        \Ensure Vector $\bar{v}$ with assembled values on all interfaces interior to the block.
        \State $\bar{v} \gets u$
        \For{axis $\alpha = 1, \dots, d$} \Comment{$d$ sequential kernel launches}
        \For{\textbf{each} interface pair $(K_L, K_R)$ normal to $\alpha$ \textbf{in parallel}}
        \State $s \gets \bar{v}|_{K_L, \text{face}_\alpha = p} + \bar{v}|_{K_R, \text{face}_\alpha = 0}$ \Comment{Statically addressed, coalesced}
        \State $\bar{v}|_{K_L, \text{face}_\alpha = p} \gets s, \quad \bar{v}|_{K_R, \text{face}_\alpha = 0} \gets s$
        \EndFor
        \EndFor
        \State \Return $\bar{v}$
    \end{algorithmic}
\end{algorithm}

% Paragraph: unstructured inter-block kernels, recalled for completeness.
\paragraph{Unstructured inter-block interfaces.}
The cascade relies on every interior vertex having the regular $2^d$-cell neighborhood of a tensor-product grid. At the
interfaces \emph{between} macro-blocks of an unstructured base mesh this regularity is lost, and no sequence of
axis-aligned pairwise exchanges visits all of its copies. For completeness we recall that these interfaces are resolved
by a second phase of dedicated kernels (Algorithm~\ref{alg:unstruct_dss}): the affected DoFs are partitioned by
codimension into three mutually disjoint sets, each processed by its own kernel: macro-faces (their bounding lines and
corners masked out), macro-lines (endpoints masked out, in 3D only), and macro-vertices. A macro-face is shared by two
blocks, so its kernel is a fixed pairwise exchange; macro-lines and macro-vertices carry a runtime valence and instead
gather their sharing set into registers, accumulate, and scatter the completed sum back to every copy. The one subtlety
is that the structured pass has already summed the several copies a line or vertex has \emph{within} each block to a
common per-block value, so the inter-block gather must add only one representative copy per sharing block, because
summing all copies would over-count by the within-block multiplicity. Face and line orientations are reconciled on the
fly by integer arithmetic, with no connectivity lookup beyond the macro-block adjacency~\cite{wichrowski2026DSS}. This
phase runs after the structured sweeps, consuming the partial block-boundary sums they produce; it is the only part of
the DSS pipeline that consults connectivity tables, and the affected DoFs form a lower-dimensional skeleton of the
mesh, so its cost is minor.

\begin{algorithm}[H]
    \caption{Unstructured inter-block summation (summary; see~\cite{wichrowski2026DSS} for details).}
    \label{alg:unstruct_dss}
    \begin{algorithmic}[1]
        \Require Vector $\bar{v}$ from Algorithm~\ref{alg:struct_dss}; disjoint sets of macro-faces $\mathcal{F}$ (pairwise), macro-lines $\mathcal{L}$, macro-vertices $\mathcal{V}$; one representative cell per sharing block flagged on each line/vertex entity.
        \Ensure Vector $\bar{v}$ with globally assembled macro-block interfaces.
        \For{\textbf{each} set $\mathcal{X} \in \{\mathcal{F}, \mathcal{L}, \mathcal{V}\}$} \Comment{Three disjoint kernel launches}
        \For{\textbf{each} entity $X \in \mathcal{X}$ \textbf{in parallel}}
        \State $s \gets \sum_{K \in \mathrm{rep}(X)} \textsc{Canonicalize}\bigl(\bar{v}|_{K, X}\bigr)$ \Comment{One representative per sharing block (both blocks, for faces)}
        \State $\bar{v}|_{K, X} \gets \textsc{InverseCanonicalize}(s)$ \textbf{ for each } $K \ni X$ \Comment{Scatter to all copies}
        \EndFor
        \EndFor
        \State \Return $\bar{v}$
    \end{algorithmic}
\end{algorithm}

% Paragraph: where the structured cascade breaks on adaptive hierarchies.
\subsubsection{Refinement edges.}
On the level meshes $\mathcal{T}_\ell$ of an adaptive hierarchy, a third situation arises that neither phase is
designed for. The structured cascade enumerates its exchange partners by memory layout, not by topology: it assumes
that across every interior interface of the block there \emph{is} a neighboring cell holding the matching partial sum.
At a refinement edge this assumption fails. Where the locally refined region of level $\ell$ ends (e.g., at re-entrant,
``L-shaped'' corners of the refined region within a macro-block, Figure~\ref{fig:dss_lshaped}), the cascade exchanges
face data with cells that have no active counterpart at the current level: the values it reads there are stale or zero,
and the sums it produces on these hanging interfaces are incorrect. Making the kernel aware of such corners would
reintroduce the per-interface topology queries that the structured design eliminates.

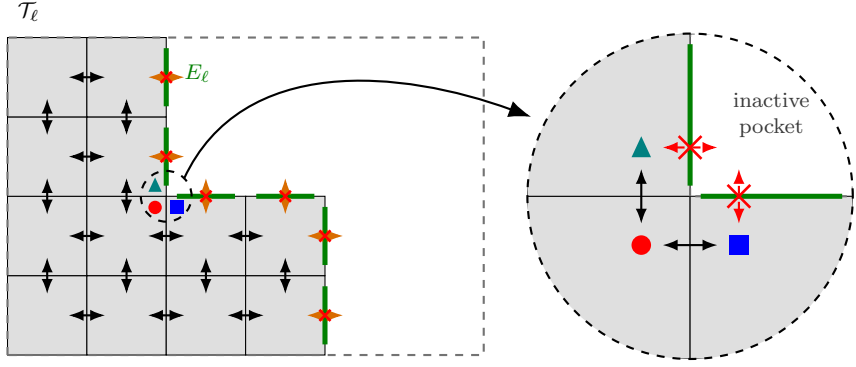
\begin{figure}[htbp]
    \centering
    % Failure of the structured DSS cascade at a re-entrant (L-shaped) corner of a
% refined region on a level mesh (flat top view). The geometry is the level mesh
% T_1 of the local-multigrid figure: a rectangular domain whose level-ell active
% cells form an L, leaving a re-entrant coarse pocket with no active level-ell cell.
% The structured cascade sweeps along rows and columns and exchanges face data
% across every interior interface -- including those on the refinement edge E_ell,
% where the neighbor is inactive. Those exchanges (orange, crossed) produce
% incorrect sums on the hanging interfaces.
\def\egap{4pt}   % gap at coarse nodes / inward shift of the corner dot
\def\agap{8pt}   % trim at each end of the cascade-exchange arrows
\def\xsz{0.07}   % half-size of the red "invalid sum" crosses
\begin{tikzpicture}[
        scale=1.05,
        act/.style={fill=gray!25, draw=black, thin},      % active fine cell
        foot/.style={draw=black!55, line width=0.8pt, dashed}, % domain outline
        redge/.style={green!50!black, line width=1.9pt,
                        shorten >=\egap, shorten <=\egap},        % refinement edge E_ell
        bad/.style={{Latex[length=2.2mm]}-{Latex[length=2.2mm]},
        orange!85!black, line width=1pt,
        shorten >=\agap, shorten <=\agap},         % offending exchange
        good/.style={{Latex[length=2mm]}-{Latex[length=2mm]},
        black, line width=0.8pt,
        shorten >=\agap, shorten <=\agap},         % valid exchange
        miss/.style={{Latex[length=2mm]}-{Latex[length=2mm]},
        red, line width=0.8pt, dashed,
        shorten >=\agap, shorten <=\agap},         % non-existent exchange
        lab/.style={font=\scriptsize}, ]

        % --- active (shaded) fine cells of the L-shaped level-ell region ---
        % bottom block [0,4]x[0,2] plus the top-left cell-strip [0,2]x[2,4]
        \foreach \c/\r in {0/0,1/0,2/0,3/0, 0/1,1/1,2/1,3/1, 0/2,1/2, 0/3,1/3}
        \filldraw[act] (\c,\r) rectangle (\c+1,\r+1);

        % --- domain (macro-block) outline, rectangular and dashed (as level T_1) ---
        \draw[foot] (0,0) rectangle (6,4);
        \node[lab, anchor=south west] at (0,4.08) {$\mathcal{T}_\ell$};

        % --- refinement edge E_ell, broken into cell-wise pieces (gap at nodes) ---
        \foreach \y in {0,1} { \draw[redge] (4,\y) -- (4,\y+1); }   % along u=4
        \foreach \x in {2,3} { \draw[redge] (\x,2) -- (\x+1,2); }   % along v=2
        \foreach \y in {2,3} { \draw[redge] (2,\y) -- (2,\y+1); }   % along u=2
        \node[lab, green!50!black, anchor=west] at (2.1,3.55) {$E_\ell$};

        % --- the three DOF markers sitting in the active cells around the
        % re-entrant corner (2,2): a dot, a square and a triangle, each just
        % inside its own cell. {cx}{cy}{off}{s} = corner, inward offset, size. ---
        \newcommand{\hangmarkers}[4]{%
                % dot (red)      -- lower-left cell
                \fill[red]  ($(#1,#2)+(-#3,-#3)$) circle (#4);
                % square (blue)  -- lower-right cell
                \fill[blue] ($(#1,#2)+(#3-#4,-#3-#4)$) rectangle ($(#1,#2)+(#3+#4,-#3+#4)$);
                % triangle (teal)-- upper-left cell
                \fill[teal] ($(#1,#2)+(-#3,#3+1.15*#4)$)
                -- ($(#1,#2)+(-#3-#4,#3-#4)$)
                -- ($(#1,#2)+(-#3+#4,#3-#4)$) -- cycle;
        }
        \hangmarkers{2}{2}{\egap}{2.4pt}
        % small circle around the cluster -> magnified in the close-up on the right
        \draw[black, thick,dashed] (2,2) circle (0.32);

        % small red cross marking an incorrect hanging-interface sum
        \newcommand{\wrong}[2]{%
                \draw[red, line width=1pt] (#1-\xsz,#2-\xsz) -- (#1+\xsz,#2+\xsz);
                \draw[red, line width=1pt] (#1-\xsz,#2+\xsz) -- (#1+\xsz,#2-\xsz);
        }
        % --- offending structured-cascade exchanges crossing E_ell into the pocket ---
        % x-sweep exchanges crossing u=2
        \draw[bad] (1.5,2.5) -- (2.5,2.5);   \wrong{2}{2.5}
        \draw[bad] (1.5,3.5) -- (2.5,3.5);   \wrong{2}{3.5}
        % y-sweep exchanges crossing v=2
        \draw[bad] (2.5,1.5) -- (2.5,2.5);   \wrong{2.5}{2}
        \draw[bad] (3.5,1.5) -- (3.5,2.5);   \wrong{3.5}{2}
        % x-sweep exchanges crossing u=4 (two bottom-right cells)
        \draw[bad] (3.5,0.5) -- (4.5,0.5);   \wrong{4}{0.5}
        \draw[bad] (3.5,1.5) -- (4.5,1.5);   \wrong{4}{1.5}

        % --- valid exchanges across interior faces between two active cells ---
        % (black, every other face) -- x-faces centred on a vertical face at u=#1
        \newcommand{\xex}[2]{\draw[good] (#1-0.5,#2) -- (#1+0.5,#2);}
        % y-faces centred on a horizontal face at v=#2
        \newcommand{\yex}[2]{\draw[good] (#1,#2-0.5) -- (#1,#2+0.5);}
        % vertical faces (x-exchanges) between horizontally adjacent active cells
        \xex{1}{0.5} \xex{2}{0.5} \xex{3}{0.5}
        \xex{1}{1.5} \xex{2}{1.5} \xex{3}{1.5}
        \xex{1}{2.5} \xex{1}{3.5}
        % horizontal faces (y-exchanges) between vertically adjacent active cells
        \yex{0.5}{1} \yex{1.5}{1} \yex{2.5}{1} \yex{3.5}{1}
        \yex{0.5}{2} \yex{1.5}{2}
        \yex{0.5}{3} \yex{1.5}{3}

        % =====================================================================
        % CLOSE-UP of the circled re-entrant corner, illustrating the failing
        % exchange: three active cells meet at the corner, the fourth quadrant
        % is the inactive pocket. The cascade still tries to sum a contribution
        % across the two pocket-facing faces (E_ell) -> incorrect.
        % =====================================================================
        \begin{scope}[shift={(8.6,2.0)}, scale=2.05]
                % clip the magnified view to a circular lens
                \begin{scope}
                        \clip (0,0) circle (1.0);
                        % three active quadrant cells (corner at local origin)
                        \filldraw[act] (-1,-1) rectangle (0, 0);   % lower-left  (dot)
                        \filldraw[act] ( 0,-1) rectangle (1, 0);   % lower-right (square)
                        \filldraw[act] (-1, 0) rectangle (0, 1);   % upper-left  (triangle)
                        % upper-right quadrant: inactive pocket (blank, dashed outline)
                        \draw[foot] (0,0) rectangle (1,1);
                        \node[lab, align=center, gray!45!black] at (0.5,0.5)
                        {inactive\\pocket};
                        % refinement edge along the two pocket-facing faces
                        \draw[redge] (0,0) -- (1,0);
                        \draw[redge] (0,0) -- (0,1);
                        % valid exchanges along the principal axes (black, solid)
                        \draw[good] (-0.30,-0.30) -- ( 0.30,-0.30);   % circle <-> square (x)
                        \draw[good] (-0.30,-0.30) -- (-0.30, 0.30);   % circle <-> triangle (y)
                        % non-existent exchanges reaching into the empty pocket (red dashed)
                        \draw[miss] (-0.30, 0.30) -- ( 0.30, 0.30);   % triangle <-> (missing)
                        \draw[miss] ( 0.30,-0.30) -- ( 0.30, 0.30);   % square   <-> (missing)
                        % cross out the invalid exchanges at their midpoints
                        \wrong{0}{0.30}
                        \wrong{0.30}{0}
                        % enlarged DOF markers (drawn last, on top of the arrows)
                        \hangmarkers{0}{0}{0.30}{0.06}
                \end{scope}
                % circular lens border
                \draw[black, thick, dashed] (0,0) circle (1.0);
        \end{scope}

        % curved "zoom" arrow from the small circle up to the enlarged lens
        \draw[-{Latex[length=3mm]}, thick]
        (2.226,2.226) to[out=65, in=160] (6.6,3.);

\end{tikzpicture}
    \caption{Interaction of the structured DSS cascade with an adaptively refined region, shown on the level mesh
        $\mathcal{T}_\ell$ whose active cells (gray) form a re-entrant L, leaving an inactive pocket with no level-$\ell$
        cell behind the refinement edge $E_\ell$ (green). The dimensionally-split passes exchange face data along each
        principal axis: between two active cells the exchange is well defined (black arrows), but along faces lying on
        $E_\ell$ the cascade still reaches toward the empty pocket, producing incorrect sums (crossed, orange). The
        close-up of the re-entrant corner makes this explicit: the three active cells meeting there carry a degree of
        freedom (dot, square, triangle), and the valid axis exchanges (black) couple them, whereas the two exchanges
        that would cross $E_\ell$ into the pocket have no partner (dashed, crossed). The masked
        smoother~\eqref{eq:cellwise_smoother} renders these values irrelevant: the residual is zeroed on $E_\ell$ before
        DSS, and the correction there is discarded by construction.}
    \label{fig:dss_lshaped}
\end{figure}

% Paragraph: ... and why we don't need to fix it.
We do not need the structured cascade to be correct on these interfaces, because the multigrid smoother does not use
the values it produces there. All offending DoFs lie, by definition, on the refinement edge $E_\ell$, and the masked
smoother~\eqref{eq:cellwise_smoother} treats $E_\ell$ specially on \emph{both} sides of the DSS step: the dual
restriction $\iota_{S_\ell}^*$ on line~3 of Algorithm~\ref{alg:smoother} zeros the residual on $E_\ell$ \emph{before}
DSS ever touches it, and the primal inclusion $\iota_{S_\ell}$ discards the correction on $E_\ell$ afterwards: those
entries are zero by construction, regardless of the cascade exchanges across the hanging interfaces. The incorrect sums
are thus produced into entries that are masked out, and never propagate into the iterate. The edge constraints required
by the local multigrid method (\secref{sec:edge_masking}) license the kernel: the solver can exploit the structured DSS
on complex adaptive patterns without requiring a topology-aware variant of Algorithm~\ref{alg:struct_dss}.

Consequently, the solver presented in this paper handles no hanging-node constraints: at no point does any component
evaluate, store, or enforce a constraint across a non-conforming interface. Continuity across refinement edges is
carried entirely by the inter-grid transfers and the residual mask of the local multigrid method.

\subsection{Outer Krylov Acceleration}
\label{sec:krylov_acceleration}

% Paragraph: the V-cycle as a stationary iteration and as a consistent CG preconditioner.
In the numerical experiments the adaptive V-cycle is used in two ways: as a stationary iteration in its own right, and
as a preconditioner for the conjugate gradient method, executed entirely on unassembled cell-wise data. All computation
is carried out in \texttt{fp64}, so the V-cycle is a \emph{fixed} linear operator that does not change between
iterations, and standard preconditioned conjugate gradients (CG) applies directly. The companion
paper~\cite{wichrowski2026DSS} proves that conjugate gradients written in primal-dual form is identical, iterate by
iterate, to its standard counterpart applied to the assembled system, provided every preconditioner application is
consistent in the sense of~\eqref{eq:precond_consistency}. The V-cycle constructed in this paper satisfies this
condition by design: it consumes the residual only through the DSS step inside the smoother and returns a continuous
correction (Theorem~\ref{thm:mg_local_equivalence}). Consequently the complete solver (CG accelerating the adaptive
cell-wise V-cycle) is equivalent to its classically assembled counterpart, while performing inter-cell communication
exclusively inside the smoothing sweeps.

\section{Implementation Notes}
\label{sec:implementation}

The V-cycle of \secref{sec:multigrid} is, on a GPU, entirely memory bound: at moderate polynomial degrees the
sum-factorized operator application streams each cell-wise vector once through memory and performs only a few tens of
arithmetic operations per loaded value, so a full operator application $A_\ell u_\ell$ costs approximately the same as
a plain vector addition. The relevant cost model for one V-cycle is therefore not the operation count but the number of
\emph{full-vector streaming passes} (plus the fixed overhead of each kernel launch). This section records the measures
by which our reference implementation reduces a V-cycle to the minimal number of such passes. None of the measures
alters the computed iterates: in exact arithmetic the optimized cycle is identical to a straightforward transcription
of Algorithms~\ref{alg:mg_local_iteration}--\ref{alg:mg_local_vcycle}, by the algebraic identities given below.

\paragraph{Detection of the refinement edge.}
The masked smoother of \secref{sec:smoother_realization} needs, on each level $\ell$, the set $E_\ell$ of
refinement-edge DoFs it zeroes on both sides of the DSS step. This is the one place where the solver must consult the
level topology; we compute it once at setup, reusing the DSS primitive itself rather than building a dedicated edge
structure. A face of a level-$\ell$ cell lies on $E_\ell$ when it is interior to the domain yet has no same-level
neighbour, indicating that the refined region ends. We read this from the authoritative per-level face connectivity
(built by coincident-vertex matching, so intra-block, rotated, and inter-block neighbours are all recorded): a face is
flagged when its same-level neighbour is absent and its boundary marker is interior, with one guard: since the
connectivity of a multi-block mesh may fail to record an inter-block link, we additionally require that no other
same-level cell share the face's vertex set.

The interiors of the surviving faces are only a first set: a refinement edge also meets the level along its own edges
and corners, and a DoF where a flagged face abuts an ordinary interior face (e.g., the re-entrant corner of an L-shaped
region, Figure~\ref{fig:dss_lshaped}) belongs to $E_\ell$ without being interior to any single flagged face. We use DSS
to find these incidences: set a marker to one on the flagged face DoFs, apply a single same-level DSS pass (the kernels
of \secref{sec:dss}), and read $E_\ell$ off as the support of the result. The marker spreads to the shared edge and
vertex DoFs the exchanges couple. The mask is thus computed by the very kernel it later guards and folded into the
stored inverse diagonal at setup, so $E_\ell$ never appears as a runtime object.

\paragraph{In-place contracts and shared buffers.}
Every vector in the cycle (the per-level correction $e_\ell$, right-hand side $b_\ell$, operator output $A_\ell
    u_\ell$, and the active-mesh residual) has a fixed shape and is allocated once at setup. All components obey an
in-place contract: operators write $A_\ell u_\ell$ into a caller-provided buffer, smoothers update the iterate in
place, and transfer kernels accumulate into the persistent level buffers. Besides avoiding copies, this is what enables
the kernel fusions below, since an operation whose destination is caller-fixed can be folded into the kernel that
already streams its operands. The smoother's one dual scratch (for $A_\ell u_\ell$ in Algorithm~\ref{alg:smoother})
need not be owned: it borrows the cycle's per-level operator-output buffer (which the V-cycle uses to form the residual
for restriction) safely, since the cycle's use of that buffer at a level brackets the smoother calls there. This
reclaims one dual vector per level, a saving that matters because the cell-wise format stores every shared DoF once per
adjacent cell, so its vectors are larger than assembled ones. The aliasing is possible only because smoother and cycle
operate on identical unassembled level vectors under one in-place contract; in a classical assembled solver the two
buffers hold different formats and cannot share storage.

\paragraph{Fused residual restriction.}
The level residual $r_\ell = b_\ell - A_\ell u_\ell$ of Algorithm~\ref{alg:mg_local_vcycle} is never materialized. The
restriction kernel takes the two operands $b_\ell$ and $A_\ell u_\ell$ separately, loads both at the same fine-cell
indices, and forms the difference per child tile immediately before the tensor-product contraction $\bigotimes_i
    \mathcal{P}_{b_i}^T$ of Algorithm~\ref{alg:restriction}. Compared with a separate subtraction this removes one
full-vector write--read round trip per level at the price of a second read in a kernel that streams the data anyway.

\paragraph{Zero initial guess in the preconditioner.}
When the V-cycle preconditions the outer CG iteration it computes a fresh correction $M^{-1}\tilde{r}$ from a zero
initial guess on every call, which the implementation specializes. With $u_{\text{act}} = 0$ the operator application
$A_{\text{act}} u_{\text{act}}$ vanishes, so the residual $\tilde{r}_{\text{act}} = \tilde{b} - A_{\text{act}}
    u_{\text{act}}$ equals $\tilde{b}$. The cycle therefore skips the initial operator apply and the residual formation,
feeding $\tilde{b}$ straight into the level distribution~\eqref{eq:level_distribution} (which only reads its source).
This saves one full operator application and the residual copy--subtract on the largest level per preconditioner call,
with no change to the iterates.

\paragraph{Child-major cell ordering.}
The in-memory ordering of a level's cells has a decisive effect on the throughput of the inter-grid transfers, though
it leaves the result unchanged. The transfer kernels (Algorithms~\ref{alg:prolongation}--\ref{alg:restriction}) couple
a coarse parent with its $2^d$ children, organized so that a single launch processes one child slot $c \in \{0, \dots,
    2^d-1\}$ across all parents at once; performance is then governed by how the slot-$c$ children of consecutive parents
are spaced. In the natural topological ``cell-major'' ordering a parent's children occupy consecutive indices (local
index $2^d e + c$ for parent $e$), so one launch strides by $2^d$ and the warp-wide accesses fragment into one
transaction per cell. We instead store the level in \emph{child-major} order, grouping cells by slot (local index $c\,
    n_{\text{parent}} + e$), so the slot-$c$ children of consecutive parents are contiguous and the launch streams them as
one coalesced run. The reordering is a one-time permutation of the level's cell list and is otherwise invisible:
operator evaluation, the smoother, and the DSS kernels are oblivious to cell order, addressing each cell through static
offsets. In our measurements it lifts the achieved bandwidth of the prolongation and restriction kernels from roughly
$300$\,GB/s to over $1500$--$1000$\,GB/s (depending on polynomial degree), putting the transfers on a streaming,
memory-bound footing with the rest of the cycle.

\paragraph{Deferred assembly in the smoother: valence-baked Jacobi.}
The most consequential reordering concerns the smoothing sweep itself. Read literally, the
smoother~\eqref{eq:cellwise_smoother} assembles the masked dual residual (one DSS pass on $\tilde r_\ell$) and then
applies the diagonal scaling, so each sweep performs its inter-cell communication on the \emph{residual}. The
implementation instead defers all assembly to a single pass over the \emph{updated iterate}. Let $\bar{S}$ denote the
\emph{averaging} variant of DSS,
\begin{equation}
    \bar{S} \;=\; \mathcal{G}_\ell \bigl(\mathcal{G}_\ell^T \mathcal{G}_\ell\bigr)^{-1} \mathcal{G}_\ell^T,
    \label{eq:averaging_dss}
\end{equation}
which assembles the duplicated copies, divides by their number, and re-injects the result; note that
$\mathcal{G}_\ell^T \mathcal{G}_\ell = \operatorname{diag}(\kappa)$ is the diagonal matrix of DoF valences $\kappa$
(the number of cells sharing each DoF), so $\bar{S}$ is the Euclidean orthogonal projection onto the
continuous subspace $\mathcal{G}_\ell \mathbb{V}_{CG,\ell}$, and $\bar{S} u = u$ for every continuous $u$. Writing
$K_{\text{cell}} = \operatorname{diag}(\mathcal{G}_\ell \kappa)$ for the cell-wise valence diagonal, a direct per-DoF
computation gives the commutation identity
\begin{equation}
    \bar{S}\, K_{\text{cell}}\, D_{\text{cell}, \ell}^{-1} \;=\; D_{\text{cell}, \ell}^{-1}\, S,
    \label{eq:valence_baking}
\end{equation}
since both sides map the copies $x_i$ of a DoF with diagonal entry $d$ and valence $\kappa$ to the common value
$d^{-1} \sum_i x_i$. Consequently, for a continuous iterate $u_\ell$ the smoothing
update~\eqref{eq:cellwise_smoother} can be rewritten as
\begin{equation}
    u_\ell + \omega\, \iota_{S_\ell} D_{\text{cell}, \ell}^{-1} S\, \iota_{S_\ell}^* \tilde{r}_\ell
    \;=\;
    \bar{S} \Bigl( u_\ell + \omega\, \iota_{S_\ell} K_{\text{cell}}\, D_{\text{cell}, \ell}^{-1} \iota_{S_\ell}^* \tilde{r}_\ell \Bigr),
    \label{eq:deferred_assembly}
\end{equation}
using $\bar{S} u_\ell = u_\ell$ for the first term and~\eqref{eq:valence_baking} for the second (the masks commute
with both diagonals and with $\bar S$ as in Lemma~\ref{lem:smoother_equivalence}). The right-hand side
of~\eqref{eq:deferred_assembly} is what the implementation executes: the bracket is a purely cell-local update of each
copy by its own raw, unassembled residual entry (no copy of $\tilde b_\ell$, no DSS on the residual) and the single
averaging pass $\bar{S}$ at the end of the sweep simultaneously restores continuity and supplies the missing assembly.
The valence factor $K_{\text{cell}}$ is folded into the stored inverse diagonal at setup time
($D_{\text{cell}}^{-1} \mapsto K_{\text{cell}} D_{\text{cell}}^{-1}$), so it never appears as a runtime operation; in
particular the inter-grid transfers remain entirely valence-free as established in
\secref{sec:residual_equivalence}, and the valence enters the solver only as a one-time rescaling of a
precomputed coefficient vector. One smoothing sweep thus consists of three kernels: the cell-wise operator
application, one fused kernel performing the masked, scaled update of~\eqref{eq:deferred_assembly}'s bracket, and one
averaging-DSS pass (the sole point of inter-cell communication, as anticipated in \secref{sec:smoother_sweep}).
Algorithm~\ref{alg:smoother_fused} shows the resulting sweep, which replaces the literal transcription of
Algorithm~\ref{alg:smoother}; the precomputed diagonal $\widehat{D}^{-1}_\ell$ absorbs the valence factor and all
masks.

\begin{algorithm}[H]
    \caption{Fused smoothing sweep realizing~\eqref{eq:deferred_assembly}.}
    \label{alg:smoother_fused}
    \begin{algorithmic}[1]
        \Require Iterate $u_\ell$, raw right-hand side $\tilde{b}_\ell$ (read-only, never assembled), relaxation $\omega$; precomputed masked diagonal $\widehat{D}^{-1}_\ell := \iota_{S_\ell} \iota_{S_\ell}^*\, K_{\text{cell}}\, D_{\text{cell}, \ell}^{-1}$ (valence and masks baked in at setup).
        \Ensure Updated iterate $u_\ell$.
        \State $v_\ell \gets A_\ell\, u_\ell$ \Comment{Cell-wise operator, into the shared level buffer}
        \State $u_\ell \gets u_\ell + \omega\, \widehat{D}^{-1}_\ell \bigl(\tilde{b}_\ell - v_\ell\bigr)$ \Comment{One fused kernel: raw per-copy Jacobi update}
        \State $u_\ell \gets \bar{S}\, u_\ell$ \Comment{Averaging DSS: restores continuity \emph{and} assembles the update}
    \end{algorithmic}
\end{algorithm}

\paragraph{Constraints baked into the diagonal.}
The masks $\iota_{S_\ell}^*$ and $\iota_{S_\ell}$ of the smoother, as well as the homogeneous Dirichlet conditions on
$\partial\Omega$, are realized by zeroing the corresponding entries of the stored inverse diagonal $K_{\text{cell}}
    D_{\text{cell}, \ell}^{-1}$ at setup: a DoF whose inverse-diagonal entry is zero receives no update, which is exactly
the action of $\iota_{S_\ell} \iota_{S_\ell}^*$ in~\eqref{eq:cellwise_smoother}. No per-sweep masking pass and no
constraint application inside the DSS kernels is needed. (The valence is ill-defined on the refinement edge $E_\ell$,
but those entries are zeroed by the mask anyway.)

\section{Numerical Results}
\label{sec:results}

We evaluate the DSS multigrid solver on the Poisson problem $-\Delta u = f$ with homogeneous Dirichlet boundary
conditions, discretized with continuous $Q_p$ elements. We report polynomial degrees $p = 2, \dots, 5$, spanning a
representative range of orders. We omit $p = 1$: the single, untuned relaxation factor is mildly too large there, so it
is the one degree for which the fixed choice degrades the contraction rate (see below). All experiments are run in
\texttt{fp64} on a single A100 80\,GB SXM. The solver is the V-cycle of \secref{sec:smoother_construction} with the
cell-wise smoother realized as in \secref{sec:smoother_realization}; a fixed number of cycles is executed without any
in-loop residual monitoring (\secref{sec:energy_measure}). The convergence measure is the energy
norm~\eqref{eq:energy_measure}, computed from local data outside the timed section; it is $\sqrt{\gamma}$ with $\gamma
    = \langle r, u\rangle$, i.e.\ a preconditioner-induced norm rather than the Euclidean residual. Unless stated otherwise
the CG stopping criterion is a reduction of this measure by $10^{-14}$, a tight tolerance chosen so that the reported
iteration counts reflect the asymptotic contraction rather than an early exit.

We consider two geometries (Figure~\ref{fig:mg_geometries}). The \emph{cube} is the Cartesian unit cube, on which the
operator is applied as the factored tensor-product symmetric sum $(D\otimes B\otimes B + B\otimes D\otimes B + B\otimes
    B\otimes D)$ scaled by the single per-element scalar $\det J\,J^{-2}$: one reference operator pair shared across all
levels, with the only per-element geometry a scalar load. The \emph{ball} is a curved domain on which the operator is
the general cell-wise Laplacian (evaluate gradients, apply the identity physics, integrate gradients) with a
per-quadrature-point $Q_k$ mapping built from the curved atlas. The two paths exercise, respectively, the affine and
the fully curved ends of the operator spectrum while sharing the same smoother, transfer, and DSS machinery. The Jacobi
diagonal is built generically for both paths by probing the level operator with unit local-DoF vectors and summing
through the DSS.

The smoother uses a fixed under-relaxation $\omega$, chosen once per geometry rather than tuned per degree: $\omega =
    0.7$ on the cube and $\omega = 0.6$ on the ball. A degree-dependent optimum could be obtained from the smoothing-factor
eigenvalues of the cell-wise iteration, but we do not tune it; the single value is mildly too large for $p = 1$, which
is why we exclude that degree as noted above.

\begin{figure}[htbp]
    \centering
    \begin{subfigure}[b]{0.45\textwidth}
        \centering
        \includegraphics[width=0.7\textwidth]{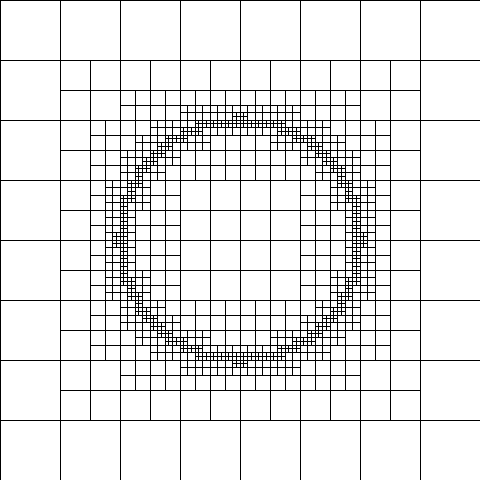}
        \caption{Cube (Cartesian)}
        \label{fig:mg_geometry_cube}
    \end{subfigure}
    \hfill
    \begin{subfigure}[b]{0.45\textwidth}
        \centering
        \includegraphics[width=0.7\textwidth]{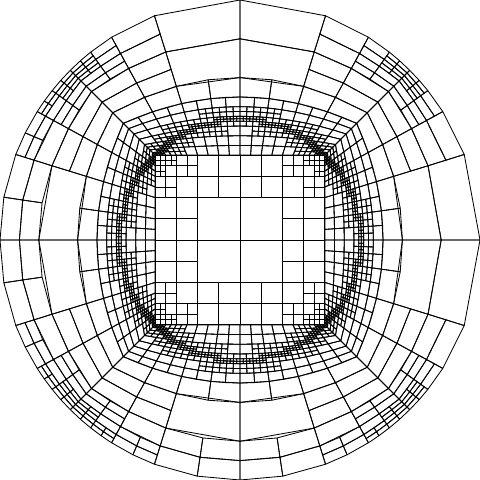}
        \caption{Ball (curved)}
        \label{fig:mg_geometry_ball}
    \end{subfigure}
    \caption{Two 2D examples of  test geometries with 4 adaptive refinement toward the domain features:
        (a) the Cartesian unit cube, evaluated with the factored tensor-product operator; (b) the curved ball,
        evaluated with the general cell-wise Laplacian and a $Q_k$ mapping.}
    \label{fig:mg_geometries}
\end{figure}

% ---------------------------------------------------------------------------
\subsection{V-cycle convergence}
\label{sec:results_convergence}

Figure~\ref{fig:mg_convergence} shows the error against the V-cycle index for $p = 2, 3, 4$, on a uniformly refined
grid (panel~a) and on an adaptively refined grid with hanging nodes (panel~b). Solid lines track the energy
measure~\eqref{eq:energy_measure}; dashed lines the Euclidean norm of the DSS-synchronized residual
$\|\mathrm{sync}(r)\|$. The two measures decay at the same asymptotic rate on both meshes, confirming that the energy
measure (available without any global assembly) is a faithful proxy for the true error and that the cell-wise iterates
remain conforming, including across the hanging-node interfaces of panel~(b).

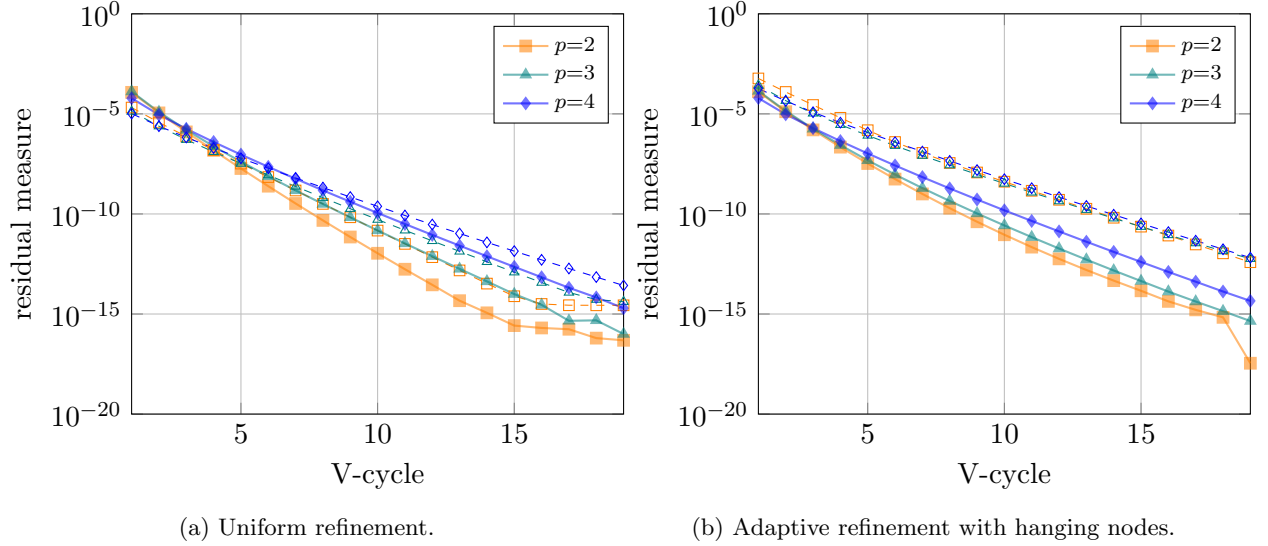
\begin{figure}[htbp]
    \centering
    % V-cycle convergence of the DSS multigrid solver, fp64. MG only (no CG).
% Two panels: (a) uniform refinement, (b) adaptive refinement with hanging
% nodes. Both [REAL DATA].
% x = V-cycle index, y = residual measure (eq:energy_measure defines
% energy = sqrt(<r,u>); we plot <r,u> itself, i.e. the energy measure squared).
% Solid lines: energy measure <r,u> (squared, sqrt not taken); dashed
% lines: ||sync(r)|| as-is (already a plain norm, not squared), the Euclidean
% norm of the DSS-synchronized residual (only meaningful on the
% uniform/conforming mesh of panel (a)).
% One colour per polynomial degree (p = 2, 3, 4).
%
% Panel (a) data: raw_results/multigrid/mg_convergence.csv, generated by
%   raw_results/multigrid/parse_convergence.py from
%   raw_results/multigrid/sweep_cube_monitor_smooth2.out (solver=mg, REF=8).
% Panel (b) data: raw_results/multigrid/mg_convergence_adaptive.csv, generated
%   by raw_results/multigrid/parse_convergence.py from
%   raw_results/multigrid/sweep_cube_adaptive_smooth2_monitor.out
%   (solver=mg, global=5, adaptive=4).
%
% This file provides the two subfigures; the outer figure environment with the
% overall caption/label lives in splitted/multigrid.tex.
\pgfplotsset{
    mgconv/.style={
            width=\linewidth, height=0.85\linewidth,
            xlabel={V-cycle}, ylabel={residual measure},
            xmin=1, xmax=19, ymin=1e-20, ymax=1,
            grid=both,
            legend style={font=\scriptsize, at={(0.97,0.97)}, anchor=north east},
            legend columns=1, legend cell align=left,
        },
}
\begin{subfigure}[b]{0.49\linewidth}
    \centering
    \begin{tikzpicture}
        \begin{semilogyaxis}[mgconv]
            \addplot[orange, mark=square*,   thick, opacity=0.55]  table[col sep=comma, x=cycle, y=p2_energy_sq] {raw_results/multigrid/mg_convergence.csv};
            \addplot[orange, mark=square,    dashed, mark options={solid}, forget plot] table[col sep=comma, x=cycle, y=p2_sync]   {raw_results/multigrid/mg_convergence.csv};
            \addplot[teal,   mark=triangle*, thick, opacity=0.55]  table[col sep=comma, x=cycle, y=p3_energy_sq] {raw_results/multigrid/mg_convergence.csv};
            \addplot[teal,   mark=triangle,  dashed, mark options={solid}, forget plot] table[col sep=comma, x=cycle, y=p3_sync]   {raw_results/multigrid/mg_convergence.csv};
            \addplot[blue,   mark=diamond*,  thick, opacity=0.55]  table[col sep=comma, x=cycle, y=p4_energy_sq] {raw_results/multigrid/mg_convergence.csv};
            \addplot[blue,   mark=diamond,   dashed, mark options={solid}, forget plot] table[col sep=comma, x=cycle, y=p4_sync]   {raw_results/multigrid/mg_convergence.csv};
            \legend{$p{=}2$, $p{=}3$, $p{=}4$}
        \end{semilogyaxis}
    \end{tikzpicture}
    \caption{Uniform refinement.}
    \label{fig:mg_conv_uniform}
\end{subfigure}
\hfill
\begin{subfigure}[b]{0.49\linewidth}
    \centering
    \begin{tikzpicture}
        \begin{semilogyaxis}[mgconv]
            \addplot[orange, mark=square*,   thick, opacity=0.55]  table[col sep=comma, x=cycle, y=p2_energy_sq] {raw_results/multigrid/mg_convergence_adaptive.csv};
            \addplot[orange, mark=square,    dashed, mark options={solid}, forget plot] table[col sep=comma, x=cycle, y=p2_sync]   {raw_results/multigrid/mg_convergence_adaptive.csv};
            \addplot[teal,   mark=triangle*, thick, opacity=0.55]  table[col sep=comma, x=cycle, y=p3_energy_sq] {raw_results/multigrid/mg_convergence_adaptive.csv};
            \addplot[teal,   mark=triangle,  dashed, mark options={solid}, forget plot] table[col sep=comma, x=cycle, y=p3_sync]   {raw_results/multigrid/mg_convergence_adaptive.csv};
            \addplot[blue,   mark=diamond*,  thick, opacity=0.55]  table[col sep=comma, x=cycle, y=p4_energy_sq] {raw_results/multigrid/mg_convergence_adaptive.csv};
            \addplot[blue,   mark=diamond,   dashed, mark options={solid}, forget plot] table[col sep=comma, x=cycle, y=p4_sync]   {raw_results/multigrid/mg_convergence_adaptive.csv};
            \legend{$p{=}2$, $p{=}3$, $p{=}4$}
        \end{semilogyaxis}
    \end{tikzpicture}
    \caption{Adaptive refinement with hanging nodes.}
    \label{fig:mg_conv_adaptive}
\end{subfigure}
    \caption{V-cycle convergence of the DSS multigrid solver, MG only (A100 80\,GB SXM, \texttt{fp64}) for
        $p = 2, 3, 4$ (one colour each). Solid lines: energy measure~\eqref{eq:energy_measure}; dashed lines:
        $\|\mathrm{sync}(r)\|$. (a) Uniform refinement (global refinement 8); (b) adaptive refinement with hanging
        nodes (global refinement 5, adaptive refinement 4).
        The energy measure vanishes at iteration zero (the initial guess is $u = 0$, so
        $\langle r, u\rangle = 0$); hence the first point of the solid curves is not shown.
    }
    \label{fig:mg_convergence}
\end{figure}

On the uniformly refined grid the V-cycle contracts the energy measure at a grid-independent rate ranging from $\approx
    0.36$ per cycle at $p = 2$ to $\approx 0.49$ at $p = 4$, so a fixed handful of cycles drives the energy error below
discretization error regardless of the number of levels. On the adaptively refined grid (panel~b) the hanging-node
interfaces are handled by the masked cell-wise smoother of \secref{sec:smoother_realization} (no constraint
application) and the contraction rate is essentially unchanged ($\approx 0.41$ at $p = 2$ and $\approx 0.50$ at $p =
    4$), demonstrating that local refinement costs nothing in iteration count.

% ---------------------------------------------------------------------------
\subsection{Solve time}
\label{sec:results_solve_time}

Table~\ref{tab:mg_solve_time} reports the throughput of the DSS multigrid as a CG preconditioner at the finest
completed problem size for each degree, on both geometries. Two throughputs are given: the per-cycle MG throughput
(unique DoFs divided by the mean V-cycle time of the un-monitored fixed-cycle loop) and the end-to-end solve throughput
(unique DoFs divided by the total CG solve time to the fixed energy-norm reduction). On the affine cube the MG cycle
sustains roughly $0.5$--$1.1$\,GDoF/s and the whole solve $45$--$100$\,MDoF/s at $8$--$11$ CG iterations; on the curved
ball, where the operator carries a per-quadrature-point mapping, the cycle runs at $0.1$--$0.25$\,GDoF/s. The cube
degree ordering is non-monotone (the $p=3$ cycle is the fastest at $1147$\,MDoF/s, ahead of $p=4$'s $662$) because the
per-cell tensor-product working set and its fit in cache and registers, not the DoF count, set the streaming rate.

Under the streaming-pass cost model, convergence monitoring is a significant expense because evaluating the residual
norm requires an extra full-vector memory pass and a DSS assembly. We therefore execute unmonitored V-cycles and rely
on the outer Krylov solver (e.g., Conjugate Gradients) to evaluate the stopping criterion using its existing inner
products.

\begin{table}[htbp]
    \centering
    % Solve-time comparison: DSS multigrid
% Data: raw_results/multigrid/mg_solve_time.csv, generated by
%   raw_results/multigrid/parse_solve_time.py from
%   raw_results/multigrid/sweep_{cube,ball}_perf_smooth2.out (2 smoothing steps).
% For each p, the finest completed solver=mg run (highest REF with a full
% "Total time" summary) is used. Unique DoFs = active leaf cells * p^3
% (marginal-rate approximation, not exact global DoF count).
% MG cycle throughput = unique DoFs / (MG per-cycle time); this comes from
% the MG run (fixed 20-cycle, un-monitored loop -- no real convergence count).
% CG iterations and Solve throughput come from the CG run at the same problem
% size (same REF): CG iterations is the genuine convergence count, and Solve
% throughput = unique DoFs / total CG solve time (the whole solution vector
% over the whole solve, not summed per-iteration throughput).
% p=1 and p=5 Ball omitted: p=1 Ball CG did not converge within the
% 20-iteration cap in this sweep; p=5 Ball has no completed run at all.
\begin{tabular}{llrrrr}
    \toprule
    $p$ & Setting & Unique DoFs   & CG iter. & MG throughput & Solve throughput \\
        &         & (estimate)    &          & [MDoF/s]      & [MDoF/s]         \\
    \midrule
    \multirow{2}{*}{$2$}
        & Cube    & 16\,777\,216  & 8        & 522.1         & 44.7             \\
        & Ball    & 14\,680\,064  & 14       & 108.2         & 7.1              \\
    \midrule
    \multirow{2}{*}{$3$}
        & Cube    & 56\,623\,104  & 9        & 1147.0        & 99.5             \\
        & Ball    & 49\,545\,216  & 15       & 252.6         & 15.4             \\
    \midrule
    \multirow{2}{*}{$4$}
        & Cube    & 134\,217\,728 & 10       & 662.4         & 56.7             \\
        & Ball    & 14\,680\,064  & 14       & 129.6         & 8.5              \\
    \midrule
    $5$ & Cube    & 262\,144\,000 & 11       & 996.5         & 77.0             \\
    \bottomrule
\end{tabular}

    \caption{Solve time of the DSS multigrid to a fixed energy-norm reduction,
        $p = 2, \dots, 5$, \texttt{fp64}, A100 80\,GB SXM.}
    \label{tab:mg_solve_time}
\end{table}

Each DSS sweep is a sequence of memory-bound streaming passes, with no local solves and no overlapping subdomain
bookkeeping. Its throughput therefore tracks the streaming bandwidth of the machine rather than any arithmetic ceiling.

We compare our approach with a vertex-patch smoother. While such smoothers achieve stronger contraction per sweep, they
require computationally expensive local subdomain solves. For $p = 3$ in \texttt{fp64}, a recent implementation reports
a solve throughput of $116$\,MDoF/s~\cite{cui2025implementation},\footnote{The cited work also reports higher
    throughputs in \texttt{fp32}, but we restrict our comparison to \texttt{fp64}. Because our method is memory-bound
    (Figure~\ref{fig:mg_roofline}), reducing the working precision would proportionally decrease memory traffic, offering a
    natural avenue for further speedups.} compared to the $99.5$\,MDoF/s achieved here on the cube geometry. The two
methods therefore perform similarly. However, this comparison favors the patch smoother, as the cited throughput
corresponds to a residual reduction of $10^{-9}$, whereas our solver reaches an energy-norm reduction of $10^{-14}$.

This is not a definitive comparison of the smoothing strategies themselves. The cited performance reflects a specific
implementation rather than the intrinsic limit of patch smoothing. Indeed, a patch smoother allows the execution of DSS
to be bypassed entirely during the smoother sweep. Combining patch smoothing with the machinery introduced here could
therefore yield further performance improvements.

% ---------------------------------------------------------------------------
\subsection{Adaptively refined meshes}
\label{sec:results_adaptive}

The masked cell-wise smoother of \secref{sec:smoother_realization} applies unchanged on adaptively refined meshes with
hanging nodes: no constraints are assembled. On a locally refined mesh the level hierarchy is markedly different from
the uniform case: the coarse levels still grow by the uniform factor $2^d = 8$, but once the adaptive front sets in the
levels grow sub-geometrically and can even shrink, so the fine part of the hierarchy is a stack of small, cheap levels
rather than a few large ones. Table~\ref{tab:mg_level_cells} lists the per-level active-cell counts of the deepest
meshes in the sweep; note the cube's level~$7$, which is smaller than level~$6$ because only a fraction of the domain
is refined that deep.

\begin{table}[htbp]
    \centering
    % Per-level active-cell counts of the deepest adaptive meshes in the sweep.
% Data: "Per-level cell counts" lines of
%   raw_results/multigrid/sweep_{cube,ball}_adaptive_smooth2.out.
% Cube: global refinement 7 + adaptive refinement 4 (11 levels);
% Ball : global refinement 5 + adaptive refinement 4 ( 9 levels).
% The first levels grow by the uniform factor 2^d = 8 (coarse, global part);
% once the adaptive front sets in the counts grow sub-geometrically and even
% dip (Cube level 7), i.e. the fine levels are small and numerous.
\begin{tabular}{@{}l*{11}{r}@{}}
    \toprule
    Level $\ell$ & $0$ & $1$ & $2$ & $3$ & $4$ & $5$ & $6$ & $7$ & $8$ & $9$ & $10$ \\
    \midrule
    Cube & 1 & 8 & 64 & 512 & 4\,096 & 32\,768 & 262\,144 & \multicolumn{1}{!{\vrule}r}{106\,944} & 385\,536 & 1\,237\,760 & 2\,468\,800 \\
    Ball & 7 & 56 & 448 & 3\,584 & 28\,672 & \multicolumn{1}{!{\vrule}r}{69\,712} & 237\,088 & 705\,568 & 1\,277\,440 & & \\
    \bottomrule
\end{tabular}

    \caption{Per-level active-cell counts of the deepest adaptively refined meshes (Cube: global refinement $7$,
        adaptive refinement $4$, $11$ levels; Ball: global refinement $5$, adaptive refinement $4$, $9$ levels). The
        coarse levels grow by the uniform factor $2^d = 8$; the fine, adaptive levels grow sub-geometrically and the
        cube's level~$7$ even dips below level~$6$. The vertical rule marks the onset of adaptive refinement: levels to
        its right are locally (adaptively) refined.}
    \label{tab:mg_level_cells}
\end{table}

One might expect these many small levels to erode throughput, since a level below the launch-saturation size cannot
fill the machine. In practice the per-cycle MG throughput on the adaptive meshes
(Table~\ref{tab:mg_adaptive_solve_time}) is on par with, and at $p = 2$ even above, the uniform-mesh throughput
(Table~\ref{tab:mg_solve_time}): the child-major coalescing of \secref{sec:smoother_realization} keeps every level
streaming, and the bulk of the DoFs still lives on the one or two largest levels, so the small levels add cycle latency
but little total time.

Table~\ref{tab:mg_adaptive_iters} reports the CG iteration count to a fixed energy-norm reduction as the number of
adaptive refinements is increased (at the coarsest global refinement of the sweep, so that the adaptive depth is the
only varying quantity). The iteration count is flat in the adaptive depth for every degree on both geometries,
confirming that local refinement costs nothing in iteration count: the hanging-node interfaces are handled by the
smoother without degrading the grid-independent contraction.

\begin{table}[htbp]
    \centering
    % CG iteration count vs. number of adaptive refinements (refinement independence).
% Data: raw_results/multigrid/sweep_{cube,ball}_adaptive_smooth2.out
%   (2 smoothing steps). Smallest global refinement in the sweep (global = 5)
%   is used so that the adaptive levels 1..4 are the only varying quantity.
% Values are the CG iteration count to the fixed energy-norm reduction with the
% DSS multigrid as preconditioner. A dash (--) marks a run absent from the sweep.
% Ball p=1 omitted (CG hit the 20-iteration cap without converging).
\subcaptionbox{Cube\label{tab:mg_adaptive_iters_cube}}{%
    \begin{tabular}{@{}c*{4}{>{\centering\arraybackslash}p{2.4em}}@{}}
        \toprule
            & \multicolumn{4}{c}{Adaptive refinements}                   \\
        \cmidrule(lr){2-5}
        $p$ & $1$                                      & $2$ & $3$ & $4$ \\
        \midrule
        $2$ & 9                                        & 10  & 9   & 9   \\
        $3$ & 10                                       & 10  & 10  & 10  \\
        $4$ & 10                                       & 10  & 10  & 10  \\
        $5$ & 11                                       & 12  & 12  & 12  \\
        \bottomrule
    \end{tabular}}
\qquad\qquad
\subcaptionbox{Ball\label{tab:mg_adaptive_iters_ball}}{%
    \begin{tabular}{@{}c*{4}{>{\centering\arraybackslash}p{2.4em}}@{}}
        \toprule
            & \multicolumn{4}{c}{Adaptive refinements}                   \\
        \cmidrule(lr){2-5}
        $p$ & $1$                                      & $2$ & $3$ & $4$ \\
        \midrule
        $2$ & 12                                       & 13  & 13  & 13  \\
        $3$ & 14                                       & 14  & 14  & 14  \\
        $4$ & 14                                       & 14  & 14  & --  \\
        \bottomrule
    \end{tabular}}

    \caption{CG iteration count of the DSS-multigrid-preconditioned solver to a fixed energy-norm reduction versus the
        number of adaptive refinements, at global refinement $5$, \texttt{fp64}, A100 80\,GB SXM. A dash marks a run
        absent from the sweep (out of memory).}
    \label{tab:mg_adaptive_iters}
\end{table}

Table~\ref{tab:mg_adaptive_solve_time} repeats the throughput measurement of Table~\ref{tab:mg_solve_time} on the most
deeply adaptively refined meshes of the sweep (adaptive refinement $4$, at the finest global refinement with completed
runs for each geometry). As on the uniform grids, the DSS multigrid sustains near the same streaming throughput despite
the deeper, more fragmented hierarchy: the cube cycle runs at $574$\,MDoF/s at $p = 2$ and $992$\,MDoF/s at $p = 3$
(versus $522$ and $1147$ uniform), and the CG iteration count is unchanged. The solve throughput tracks the MG cycle
throughput up to the fixed CG iteration count.

\begin{table}[htbp]
    \centering
    % Solve-time comparison of the DSS multigrid on adaptively refined meshes.
% Data: raw_results/multigrid/sweep_{cube,ball}_adaptive_smooth2.out
%   (2 smoothing steps). Highest adaptive-refinement level in the sweep
%   (adaptive = 4) is used, at the highest global refinement with completed
%   runs for each geometry: Cube global = 7, Ball global = 5.
% Unique DoFs = active leaf cells * p^3 (marginal-rate approximation, not the
% exact constrained global DoF count).
% MG throughput = unique DoFs / (MG per-cycle time) from the solver=mg run
%   (fixed 20-cycle, un-monitored loop). Solve throughput = unique DoFs /
%   total CG solve time from the solver=cg run at the same problem size; CG
%   iterations is the genuine convergence count from that run.
% p=1 omitted from both geometries; Ball p>=4 and Cube p=5 omitted (no
% completed adaptive=4 run at that size).
\begin{tabular}{llrrrr}
    \toprule
    $p$ & Setting & Unique DoFs   & CG iter. & MG throughput & Solve throughput \\
        &         & (estimate)    &          & [MDoF/s]      & [MDoF/s]         \\
    \midrule
    \multirow{2}{*}{$2$}
        & Cube    & 31\,490\,432  & 9        & 574.2         & 52.3             \\
        & Ball    & 15\,961\,680  & 13       & 108.5         & 7.6              \\
    \midrule
    \multirow{2}{*}{$3$}
        & Cube    & 106\,280\,208 & 9        & 992.3         & 88.8             \\
        & Ball    & 54\,870\,858  & 14       & 244.2         & 15.8             \\
    \midrule
    \multirow{1}{*}{$4$}
        & Cube    & 251\,923\,456 & 10       & 605.9         & 51.8             \\
    \bottomrule
\end{tabular}

    \caption{Solve time of the DSS multigrid on adaptively refined meshes (adaptive refinement $4$; Cube at global
        refinement $7$, Ball at global refinement $5$), \texttt{fp64}, A100 80\,GB SXM. Unique DoFs
        $\approx$ active leaf cells $\times\, p^{3}$.}
    \label{tab:mg_adaptive_solve_time}
\end{table}

% ---------------------------------------------------------------------------
\subsection{Roofline of the V-cycle building blocks}
\label{sec:results_roofline}

Figure~\ref{fig:mg_roofline} places the kernel families that make up a V-cycle (operator evaluation, restriction,
prolongation, and the cell-wise smoother) on the A100 roofline, with one point per degree $p = 2, \dots, 5$ in each
family. All are memory bound and sit on the bandwidth slope: the V-cycle moves data, it does not compute, and its cost
is governed by the bytes streamed rather than the FLOPs issued.

\begin{figure}[htbp]
    \centering
    \input{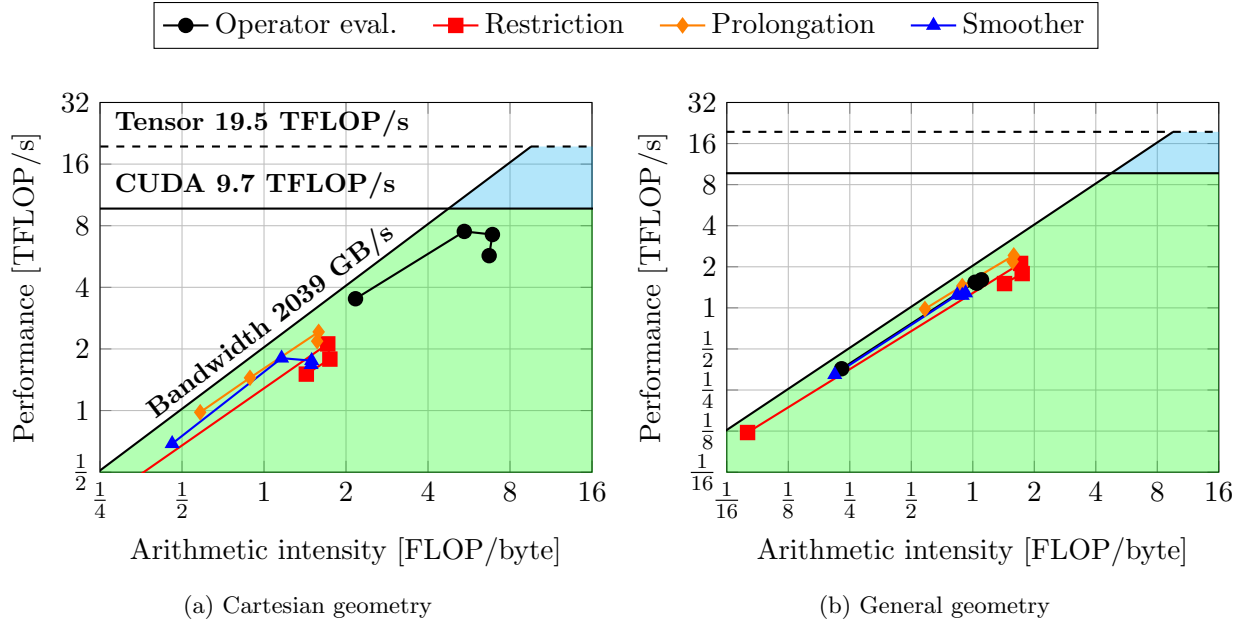}
    \caption{Roofline of the multigrid building blocks on the A100 80\,GB SXM (\texttt{fp64}): operator
        evaluation, restriction, prolongation, and the cell-wise smoother, one point per degree
        $p = 2, \dots, 5$ connected within each family. Solid lines are the bandwidth slope (2039\,GB/s) and the
        vector FP64 ceiling (9.7\,TFLOP/s); the dashed line is the tensor-core FP64 ceiling (19.5\,TFLOP/s).
        Panel (a) is the Cartesian geometry, panel (b) a general geometry; transfers are geometry-independent
        and shared between the two.}
    \label{fig:mg_roofline}
\end{figure}

Because every component is bandwidth bound, the time per V-cycle is the sum of the bytes moved by each kernel divided
by the achievable bandwidth. Comparing the two panels makes the dominant cost explicit: on a general geometry
(panel~(b)) the operator evaluation and the smoother drop well below their Cartesian counterparts (panel~(a)) and
collapse onto the bandwidth slope, because the per-cell geometry data (the Jacobians and metric terms that must be
streamed alongside the solution) dwarf the degree-of-freedom data and consume the entire memory bandwidth. On the
Cartesian grid this geometry information is trivial and the kernels climb toward the arithmetic ceilings; on a curved
mesh it is not, and the transfers, which carry no geometry, are unaffected and therefore identical between the two
panels.

This identifies where the remaining headroom lies. Since the bottleneck is the geometry traffic rather than the
arithmetic, the lever on solve time is a smoother with better data locality that reuses geometry data across its
sweeps~\cite{wichrowski2025local,wichrowski2025smoothers}, or matrix-free formulations that avoid materializing the
per-cell metric terms altogether through matching or geometry-aware
methods~\cite{cui2025multigrid,wichrowski2025geometric,wichrowski2026SBM}. This is the sense in which the method is
already ``free'' on Cartesian grids: there is no arithmetic headroom being wasted, and the cell-wise storage and
structured DSS path of Part~I already minimize the solution traffic, so what remains is the geometry.

\section{Conclusions}
\label{sec:conclusions}

We have presented a geometric multigrid preconditioner for high-order continuous finite elements that runs entirely on
the redundant, cell-wise storage of the companion framework~\cite{wichrowski2026DSS}: the assembled global vector is
never formed on any level of the hierarchy, and the single inter-cell primitive of the whole V-cycle is one direct
stiffness summation inside the smoother. The central observation is that in this storage paradigm the machinery that
ordinarily complicates adaptive multigrid dissolves. Hanging-node constraints are never assembled: the plain
tensor-product transfers, applied to the \emph{unassembled} residual, reproduce the classical constrained restriction
algebraically, including the action of the transposed constraint matrix, and the edge operators of local smoothing
reduce to a pointwise masking of the residual. We proved that the resulting cell-wise V-cycle is equivalent, iterate by
iterate, to the classical local multigrid method, so it inherits that method's convergence theory rather than requiring
a new one.

The numerical experiments bear this out. On both a Cartesian cube and a curved ball, in uniform and in adaptively
refined form, the V-cycle contracts the energy measure at a grid-independent rate, and, the contraction is unchanged
between the uniform and the hanging-node meshes: local refinement costs almost nothing in iteration count, because the
interfaces are handled by the masked smoother with no constraint application. On the GPU the whole solve is memory
bound; using nothing more than a masked point-Jacobi smoother it reaches a solve throughput at $p = 3$ in \texttt{fp64}
on par with a recent patch-smoother implementation~\cite{cui2025implementation}, despite driving the error to a
substantially tighter tolerance.

The roofline analysis locates the remaining headroom. Every kernel of the V-cycle sits on the bandwidth slope, and the
comparison of the two geometries shows that on curved meshes it is the per-cell geometry traffic, namely the Jacobians
and metric terms, that dwarfs the degree-of-freedom data and consumes the memory bandwidth, while on Cartesian grids
that information is trivial and the operator is effectively free. The lever on solve time is therefore not more
arithmetic but less geometry traffic: a smoother with better data locality that reuses geometry data across its
sweeps~\cite{wichrowski2025local,wichrowski2025smoothers}, or matrix-free formulations that avoid materializing the
metric terms through matching or geometry-aware
methods~\cite{cui2025multigrid,wichrowski2025geometric,wichrowski2026SBM}. Combining the constraint- and valence-free
adaptive machinery developed here with such smoothers, and reducing the working precision~\cite{goddeke2007performance}
now that the method is demonstrably bandwidth bound, are the natural directions for further gains.

\paragraph{Declarations}
Language models (Claude, Gemini) were applied as smoothers during drafting; the core(coarse) content as well as
transfers from those drafts are the author's own. The author retains full accountability for all scientific content.

%%%%%%%%%%%%%%%%%%%%%%%%
% Numerical results — Laplace operator.
%%%%%%%%%%%%%%%%%%%%%%%%

\bibliographystyle{siam}
\bibliography{literature,added_literature}

\end{document}